\documentclass[11pt]{article}
\usepackage{amssymb}
\usepackage{colortbl}
\usepackage{amsfonts,amsmath, longtable}

        \def\theequation{\thesection.\arabic{equation}}

\newcommand{\mF}{{\mathcal F}}
\newcommand{\mR}{{\mathcal R}}
\newcommand{\mH}{{\mathcal H}}

\newcommand{\vf}{\varphi}
\newcommand{\al}{\alpha}
\newcommand{\be}{\beta}

\newcommand{\om}{\omega}
\newcommand{\vth}{\vartheta}

\newcommand{\MatM}{ {\rm Mat}(M,\mathbb C) }

\newcommand{\mC}{\mathbb C}
\newcommand{\mZ}{\mathbb Z}

\newcommand{\ka}{\kappa}

\newcommand{\z}{{\zeta}}

\newtheorem{predl}{Proposition}[section]
\newtheorem{lemma}{Lemma}[section]

\newenvironment{proof}{\par\noindent{\bf Proof.}}{\hfill$\scriptstyle\blacksquare$}

\def\beq{\begin{equation}}
\def\eq{\end{equation}}
\def\p{\partial}

\newtheorem{theor}{Theorem}

\newcommand{\mats}[4]{\left(\begin{array}{cc}{#1}&{#2}\\ {#3}&{#4}
\end{array}\right)}

\def\res{\mathop{\hbox{Res}}\limits}

\begin{document}

\setcounter{page}{1}

\begin{center}

\

\vspace{-0mm}




{\Large{\bf Anisotropic spin  generalization of elliptic    }}

\vspace{3mm}

{\Large{\bf Macdonald-Ruijsenaars operators and R-matrix identities }}

 \vspace{15mm}

 {\Large {M. Matushko}}
\qquad\quad\quad
 {\Large {A. Zotov}}


  \vspace{5mm}

{\em Steklov Mathematical Institute of Russian
Academy of Sciences,\\ Gubkina str. 8, 119991, Moscow, Russia}

   \vspace{3mm}

 {\small\rm {e-mails: matushko@mi-ras.ru, zotov@mi-ras.ru}}

\end{center}

\vspace{0mm}

\begin{abstract}
We propose commuting set of matrix-valued difference operators in terms
 of the elliptic Baxter-Belavin $R$-matrix in the fundamental representation of
${\rm GL}_M$. In the scalar case $M=1$ these operators are the elliptic Macdonald-Ruijsenaars operators, while
in the general case they can be viewed as anisotropic versions of the quantum spin Ruijsenaars Hamiltonians.
We show that commutativity of the operators for any $M$
is equivalent to a set of $R$-matrix identities. The proof of identities is based on the properties of elliptic $R$-matrix including the quantum and the associative Yang-Baxter equations. As an application of our results, we introduce elliptic generalization of
q-deformed Haldane-Shastry model.

\end{abstract}

%

{\small{
\tableofcontents
}}

\subsection*{\ \ \ \ \ \ \underline{Main notations:}}

\ \vspace{-7mm}

{\small{

$N$ -- number of complex variables $z_1,...,z_N$ in operators;

$M$ -- rank of ${\rm GL}_M$ $R$-matrix, i.e. the size of the basis matrices (\ref{a971}). $M=1$ is called the scalar case;

$k$ -- integer number used for numeration of identities (and operators);

$\eta$ -- complex variable (constant parameter) entering the shift operators $p_i=\exp(-\eta\p_{z_i})$;

$\hbar$ -- complex variable (constant parameter) entering quantum $R$-matrix definition;

$\Gamma$ -- lattice of periods $\mZ\oplus\mZ\tau$ of elliptic curve;

$\tau$ -- moduli of elliptic curve $\mC/\Gamma$, ${\rm Im}(\tau)>0$;

$\phi(z)$ -- the Kronecker elliptic function $\phi(\hbar,z)$ (\ref{q01})-(\ref{q02});

$p_i$ -- the shift operator (\ref{p_i});

${\bf p}_I$ -- product of shift operators (\ref{pI}) over subset $I\subset\{1,...,N\}$;

$(I,J)$ -- product (\ref{IJ}) of $\phi$-functions;

$D_k$ -- the Macdonald-Ruijsenaars operators (\ref{Dscalar}) or (\ref{D+});

${\mathcal H}$ -- vector space ${\mathcal H}=(\mC^M)^{\otimes N}$;

$R^\hbar_{ij}(z)$ -- $R$-matrix, i.e. an element of ${\rm End}({\mathcal H})$ acting on $i$-th and $j$-th tensor components of ${\mathcal H}$;

$R_{ij}$ and $R^-_{ij}$ -- short notations for $R^\hbar_{ij}(z_i-z_j)$ and $R^\hbar_{ij}(z_i-z_j-\eta)$ respectively;

${\bar R}_{ij}$ --  $R$-matrix normalized as in (\ref{q04})-(\ref{q05});

$P_{ij}$ -- the matrix permutation operator (\ref{P12}) permuting tensor components in $\mH$;

$\overrightarrow{\prod\limits^{N}_ {j=1}} R_{ij}$ and $\overleftarrow{\prod\limits^{N}_ {j=1}} R_{ij}$ -- the arrows mean the ordering
$\overrightarrow{\prod\limits^{N}_ {j=1}} R_{ij}=R_{i1}R_{i2}...R_{iN}$ and $\overleftarrow{\prod\limits^{N}_ {j=1}} R_{ij}=R_{iN}R_{i,N-1}...R_{i1}$;

$\mR_{I,J}$ and $\mR_{I,J}'$ -- the products of $R$-matrices (\ref{RIJ1})
and (\ref{RIJ'1}) respectively;

${\mathcal D}_k$ -- the (anisotropic) spin operators (\ref{q10}) or (\ref{Dspin2}).

}}

\newpage
\section{Introduction}
\setcounter{equation}{0}

The elliptic Macdonald-Ruijsenaars operators were introduced in \cite{Ruij} generalizing the (trigonometric) Macdonald operators \cite{Macd}. These  operators can be viewed as a set of commuting Hamiltonians in quantum many-body system of interacting particles (the elliptic Ruijsenaars-Schneider model). Commutativity of operators can be established using several different approaches including the Dunkl-Cherednik operators \cite{Chered,BFV,KH}, the quantization by dynamical $R$-matrices \cite{FV}, the quantization via non-dynamical $R$-matrix \cite{Has}, the quantization through the quantum Lax pairs \cite{Chalykh} and the group-theoretical approach \cite{ACF,AKO}. The studies of elliptic Macdonald-Ruijsenaars operators and their generalizations attract much attention. Some recent results can be found in \cite{MMZ,VDG}.

In this paper we propose anisotropic version for elliptic spin Macdonald-Ruijsenaars operators and prove their commutativity following the original approach of \cite{Ruij}. Namely, our strategy is to deduce the underlying identities and then prove them.
Compared to \cite{Ruij} (in what follows we refer to it as the scalar case)
the derivation and proof of the identities in our case is much more complicated since the elliptic functions are replaced by ${\rm GL}_M$ $R$-matrices. In this respect we also use the approach of A. Polishchuk \cite{Pol}, where the elliptic Baxter-Belavin $R$-matrix \cite{Baxter,Belavin} was shown to satisfy the associative Yang-Baxter equation \cite{FK}. Then the ${\rm GL}_M$
elliptic $R$-matrix is considered as matrix valued function on elliptic curve, and it turns into a certain function (the Kronecker elliptic function) when $M=1$. The proof of commutativity of spin operators uses not only the quantum Yang-Baxter equation (and unitarity condition) since the latter is identically fulfilled in $M=1$ case, while the identities remain non-trivial in the scalar case.   In this way our results provide non-commutative (matrix) generalization of \cite{Ruij}.

\paragraph{Elliptic Macdonald-Ruijsenaars operators.}
For $i=1,\dots,N$ denote by $p_i$ the shift operator acting on function $f(z_1,\dots ,z_N)$ as follows:
\beq\label{p_i}
(p_if)(z_1,z_2,\dots z_N)=\exp\left(-\eta \frac{\partial}{\partial z_i}\right)f(z_1,\dots,z_N)=f(z_1,\dots,z_i-\eta,\dots, z_N).
\eq
Define the Kronecker elliptic function on elliptic curve $\mC/\Gamma$, $\Gamma=\mZ\oplus\mZ\tau$ with
moduli $\tau$ (${\rm Im}(\tau)>0$)  \cite{Weil}:
\beq\label{q01}
\displaystyle{
    \phi(x, y) =
    \frac{\vartheta'(0) \vartheta (x + y)}{\vartheta (x) \vartheta (y)}\,,
}
\eq
and denote
\beq\label{q02}
\displaystyle{
    \phi(z)=\phi(\hbar,z)\,.
}
\eq
Necessary properties and definitions of elliptic functions\footnote{The term elliptic function not necessarily means double-periodic function in this paper.} are given in the Appendix A.

In \cite{Ruij} S.N.M. Ruijsenaars introduced the following set
of operators\footnote{In fact, in \cite{Ruij} the definition of $\phi(z)$
is different from (\ref{q01}) but it can be transformed to (\ref{q01}) by simple redefinition.}:
\begin{equation}\label{Dscalar}
 D_k=\sum\limits_{\substack{|I|=k}}\prod\limits_{\substack{i\in I \\ j\notin I}}\phi(z_j-z_i)\prod_{i\in I}p_{i},\qquad k=1,\dots,N,
\end{equation}
where the sum is taken over all subsets $I$ of $\{1,\dots,N\}$ of size $k$. In the trigonometric limit  (together with some simple redefinitions) the expressions (\ref{Dscalar}) turn into
the Macdonald operators \cite{Macd}:
\begin{equation}\label{Macd}
 D_k^{Macd}=t^{\frac{k(k-N)}{2}}\sum\limits_{\substack{|I|=k}}\prod\limits_{\substack{i\in I \\ j\notin I}}\frac{tx_i-x_j}{x_i-x_j}\prod_{i\in I}q^{x_i\p_{x_i}},\qquad k=1,\dots,N
\end{equation}
with $t=\exp (-2\pi\imath\hbar)$, $x_k=\exp (2\pi\imath z_k)$ and $q=\exp (-\eta)$.

  It was shown in \cite{Ruij} that $D_k$ pairwise commute iff the function $\phi(z)$ satisfies the set of functional equations:
\begin{equation}\label{IdenRuij}
\sum_{|I|=k}\left(\prod\limits_{\substack{i\in I\\ j\notin I}}\phi(z_j-z_i)\phi(z_i-z_j-\eta)-\prod\limits_{\substack{i\in I\\ j\notin I}}\phi(z_i-z_j)\phi(z_j-z_i-\eta)\right)=0\,,\quad
k=1,\dots,N\,.
\end{equation}
It was proved that the function $\phi(z)$ (\ref{q02}) satisfies (\ref{IdenRuij}).  Moreover, $\phi(z)$ is determined
to be of the form (\ref{q02})
due to (\ref{IdenRuij}). In this paper we consider (\ref{IdenRuij}) not as a functional equation but as a set of identities for the function $\phi(z)$, which is fixed as (\ref{q02}).

\paragraph{$R$-matrices and Yang-Baxter equations.}
Let ${\mathcal H}$ be a vector space ${\mathcal H}=(\mC^M)^{\otimes N}$.
An $R$-matrix
$R^\hbar_{ij}(z)\in {\rm End}({\mathcal H})$ is a linear map acting non-trivially in the $i$-th and $j$-th tensor components of ${\mathcal H}$ only. See (\ref{BB}) and (\ref{BB2}) for elliptic $R$-matrix.

By definition, any quantum $R$-matrix satisfies the quantum Yang-Baxter equation (QYB):
\beq\label{QYB}
\begin{array}{c}
\displaystyle{
    R^{\hbar}_{12}(u)  R^{\hbar}_{13}(u+v) R^{\hbar}_{23}(v) =
      R^{\hbar}_{23}(v) R^{\hbar}_{13}(u+v) R^{\hbar}_{12}(u)\,.
      }
\end{array}\eq
The action of permutation operators (\ref{P12})-(\ref{P12-1}) on (\ref{QYB}) provides the set of relations:
\beq\label{QYB2}
\begin{array}{c}
\displaystyle{
    R^{\hbar}_{ij}(u)  R^{\hbar}_{ik}(u+v) R^{\hbar}_{jk}(v) =
      R^{\hbar}_{jk}(v) R^{\hbar}_{ik}(u+v) R^{\hbar}_{ij}(u)
      }
\end{array}\eq
for any distinct integers $1\leq i,j,k\leq N$. Also,
\beq\label{QYB3}
\begin{array}{c}
\displaystyle{
    [R^{\hbar}_{ij}(u), R^{\hbar'}_{kl}(v)]=0
      }
\end{array}\eq
for any distinct integers $1\leq i,j,k,l\leq N$.

In this paper we deal with the elliptic Baxter-Belavin $R$-matrix \cite{Baxter,Belavin} given by (\ref{BB}).
It satisfies (\ref{QYB}) and  the unitarity property:
\beq\label{q03}\begin{array}{c}
    R^{\hbar}_{ij}(z) R^\hbar_{ji}(-z)= {\rm Id}\, \phi(\hbar,z)\phi(\hbar,-z)\,,
\end{array}\eq
where ${\rm Id}=1_{M^N}$ is the identity matrix in ${\rm End}({\mathcal H})$.
In what follows we also use $R$-matrices ${\bar R}^{\hbar}_{ij}(z)$, which are related to $R^{\hbar}_{ij}(z)$ through
\beq\label{q04}
\begin{array}{c}
    R^{\hbar}_{ij}(z) = \phi(\hbar,z){\bar R}^{\hbar}_{ij}(z)\,.
\end{array}\eq
Then
\beq\label{q05}
\begin{array}{c}
    {\bar R}^{\hbar}_{ij}(z) {\bar R}^\hbar_{ji}(-z)= {\rm Id}\,.
\end{array}\eq

It was shown in \cite{Pol} that the elliptic $R$-matrix satisfies not only
the QYB (\ref{QYB}) but also the associative Yang-Baxter equation (AYBE) \cite{FK}:
\beq\label{AYBE}
\begin{array}{c}
    R^{x}_{12} R^{y}_{23} = R^{y}_{13} R^{x-y}_{12} + R^{y-x}_{23} R^{x}_{13}, \quad R^\hbar_{ab} = R^\hbar_{ab}(z_a-z_b)\,.
\end{array}\eq

Consider the scalar case $M=1$. Then an $R$-matrix becomes a function. The elliptic $R$-matrix turns into the Kronecker
function (\ref{q02}). While the QYB (\ref{QYB}) becomes trivial (it is valid for any function), the AYBE
is a non-trivial equation given by the genus one Fay identity (\ref{Fay}). In this respect the elliptic $R$-matrix can be viewed as a matrix generalization of the Kronecker elliptic function. Being considered as a function of $z$ with simple pole at $z=0$, the Kronecker function is fixed by residue (\ref{a095}) and the quasi-periodic properties (\ref{a096}) on the lattice of periods $\Gamma=\mZ\oplus\mZ\tau$ of elliptic curve. Similarly, the elliptic $R$-matrix
(\ref{BB}) considered as a matrix valued function of $z$ with simple pole at $z=0$ is fixed by (\ref{r05}) and (\ref{r721}).

\paragraph{Purpose of the paper} is to construct commutative set of spin type generalizations of (\ref{Dscalar}) using
the elliptic ${\rm GL}_M$ $R$-matrix (\ref{BB}), which
satisfies equations (\ref{QYB})-(\ref{AYBE}).

For any $M$ define the set of spin Macdonald-Ruijsenaars operators as
\beq\label{q10}
\begin{array}{c}
  \displaystyle{
    {\mathcal D}_k=\sum\limits_{1\leq i_1<...<i_k\leq N}\left(\!\prod\limits^{N}_{\hbox{\tiny{$ \begin{array}{c}{ j=1 }\\{ j\!\neq\! i_1...i_{k-1} } \end{array}$}}}\!\phi(z_j-z_{i_1})\ \phi(z_j-z_{i_2})
    \ \cdots\
    \phi(z_j-z_{i_k})\right)\times
    }
    \\ \ \\
    \displaystyle{
    \times\left(
   \overleftarrow{\prod\limits_{j_1=1}^{i_1-1}} \bar{R}_{j_1 i_1}
   \overleftarrow{\prod\limits^{i_2-1}_{\hbox{\tiny{$ \begin{array}{c}{ j_2=1 }\\{ j_2\!\neq\! i_1 } \end{array}$}}}} \bar{R}_{j_2 i_2}
      \ \ldots\
 \overleftarrow{\prod\limits^{i_k-1}_{\hbox{\tiny{$ \begin{array}{c}{ j_k=1 }\\{ j_k\!\neq\! i_1...i_{k-1} } \end{array}$}}}} \bar{R}_{j_k i_k}
  \right)\times
   }
   \\ \ \\
     \displaystyle{
       \times p_{i_1}\cdot p_{i_2}\cdots p_{i_k}\times\left(
   \overrightarrow{\prod\limits^{i_k-1}_{\hbox{\tiny{$ \begin{array}{c}{ j_k\!=\!1 }\\{ j_k\!\neq\! i_{1}...i_{k-1}} \end{array}$}}}}\bar{R}_{i_k j_k}
   \overrightarrow{\prod\limits^{i_{k-1}-1}_{\hbox{\tiny{$ \begin{array}{c}{ j_{k-1}\!=\!1 }\\{ j_{k-1}\!\neq\! i_{1}...i_{k-2}} \end{array}$}}}}\bar{R}_{i_{k-1} j_{k-1}}
      \ \ldots\
  \overrightarrow{\prod\limits^{i_{1}-1}_ {j_1=1}} \bar{R}_{i_{1} j_{1}}\right),
 }
\end{array}\eq
where $k=1,...,N$ and $\bar{R}_{ij}=\bar{R}_{ij}^\hbar(z_i-z_j)$. In the scalar case (\ref{q10}) coincides with (\ref{Dscalar}).
\begin{theor}\label{th1}
 The operators ${\mathcal D}_k$ (\ref{q10}) commute with each other
 \beq\label{a205}
  \begin{array}{c}
  \displaystyle{
   [{\mathcal D}_k,{\mathcal D}_l]=0\quad k,l=1,...,N
 }
 \end{array}
 \eq
 iff the following set of identities holds true:
 \beq\label{a20}
  \begin{array}{c}
  \displaystyle{
   \sum\limits_{1\leq i_1<...<i_k\leq N}
    \Big({\mathcal F}^-_{i_1,...,i_k}(k,N)-{\mathcal F}^+_{i_1,...,i_k}(k,N)\Big)=0\,,
 }
 \end{array}
 \eq
where
 \beq\label{a21}
  \begin{array}{c}
  \displaystyle{
   {\mathcal F}^+_{i_1,...,i_k}(k,N)=
   \overrightarrow{\prod\limits_{l_k=i_k+1}^N} R_{i_k l_k}
   \overrightarrow{\prod\limits^N_{\hbox{\tiny{$ \begin{array}{c}{ l_{k-1}\!=\!i_{k-1}\!+\!1 }\\{ l_{k-1}\!\neq\! i_k } \end{array}$}}}}R_{i_{k-1} l_{k-1}}
      \ \ldots\
   \overrightarrow{\prod\limits^N_{\hbox{\tiny{$ \begin{array}{c}{ l_{1}\!=\!i_{1}\!+\!1 }\\{ l_{1}\!\neq\! i_2...i_k } \end{array}$}}}}R_{i_{1} l_{1}}
   \times
   }
   \end{array}
 \eq
   $$
   \begin{array}{c}
     \displaystyle{
 \times
 \overleftarrow{\prod\limits^N_{\hbox{\tiny{$ \begin{array}{c}{ j_{1}\!=\!1 }\\{ j_{1}\!\neq\! i_1...i_k } \end{array}$}}}}R_{j_1i_1}^-
 \overleftarrow{\prod\limits^N_{\hbox{\tiny{$ \begin{array}{c}{ j_{2}\!=\!1 }\\{ j_{2}\!\neq\! i_1...i_k } \end{array}$}}}}R_{j_2i_2}^-
 \ \ldots\
  \overleftarrow{\prod\limits^N_{\hbox{\tiny{$ \begin{array}{c}{ j_{k}\!=\!1 }\\{ j_{k}\!\neq\! i_1...i_k } \end{array}$}}}}R_{j_ki_k}^-
  \times
    }
   \\ \ \\
     \displaystyle{
       \times
   \overrightarrow{\prod\limits^{i_k-1}_{\hbox{\tiny{$ \begin{array}{c}{ m_k\!=\!1 }\\{ m_k\!\neq\! i_{1}...i_{k-1}} \end{array}$}}}}R_{i_k m_k}
   \overrightarrow{\prod\limits^{i_{k-1}-1}_{\hbox{\tiny{$ \begin{array}{c}{ m_{k-1}\!=\!1 }\\{ m_{k-1}\!\neq\! i_{1}...i_{k-2}} \end{array}$}}}}R_{i_{k-1} m_{k-1}}
      \ \ldots\
  \overrightarrow{\prod\limits^{i_{1}-1}_ {m_1=1}} R_{i_{1} m_{1}}
 }
 \end{array}
 $$
and
 \beq\label{a22}
  \begin{array}{c}
  \displaystyle{
   {\mathcal F}^-_{i_1,...,i_k}(k,N)=
   \overleftarrow{\prod\limits_{m_1=1}^{i_1-1}} R_{m_1 i_1}
   \overleftarrow{\prod\limits^{i_2-1}_{\hbox{\tiny{$ \begin{array}{c}{ m_2=1 }\\{ m_2\!\neq\! i_1 } \end{array}$}}}} R_{m_2 i_2}
      \ \ldots\
 \overleftarrow{\prod\limits^{i_k-1}_{\hbox{\tiny{$ \begin{array}{c}{ m_k=1 }\\{ m_k\!\neq\! i_1...i_{k-1} } \end{array}$}}}} R_{m_k i_k}
   \times
   }
   \end{array}
 \eq
   $$
   \begin{array}{c}
     \displaystyle{
 \times
 \overrightarrow{\prod\limits^N_{\hbox{\tiny{$ \begin{array}{c}{ j_{k}\!=\!1 }\\{ j_{k}\!\neq\! i_1...i_k } \end{array}$}}}}R_{i_k j_k}^-
 \overrightarrow{\prod\limits^N_{\hbox{\tiny{$ \begin{array}{c}{ j_{k-1}\!=\!1 }\\{ j_{k-1}\!\neq\! i_1...i_k } \end{array}$}}}}R_{i_{k-1} j_{k-1}}^-
 \ \ldots\
  \overrightarrow{\prod\limits^N_{\hbox{\tiny{$ \begin{array}{c}{ j_{1}\!=\!1 }\\{ j_{1}\!\neq\! i_1...i_k } \end{array}$}}}}R_{i_1 j_1}^-
  \times
    }
   \\ \ \\
     \displaystyle{
       \times
   \overleftarrow{\prod\limits^{N}_{\hbox{\tiny{$ \begin{array}{c}{ l_1\!=\!i_1\!+\!1 }\\{ l_1\!\neq\! i_{2}...i_{k}} \end{array}$}}}}R_{l_1 i_1}
    \overleftarrow{\prod\limits^{N}_{\hbox{\tiny{$ \begin{array}{c}{ l_2\!=\!i_2\!+\!1 }\\{ l_2\!\neq\! i_{3}...i_{k}} \end{array}$}}}}R_{l_2 i_2}
      \ \ldots\
  \overleftarrow{\prod\limits^{N}_ {l_k=i_k+1}} R_{l_k i_k}\,,
 }
 \end{array}
 $$
and $R_{ij}=R_{ij}^\hbar(z_i-z_j)$, $R_{ij}^-=R_{ij}^\hbar(z_i-z_j-\eta)$. In the scalar case $M=1$ identities (\ref{a20})
coincide with (\ref{IdenRuij}).
\end{theor}
Let us notice that the operators (\ref{q10}) are written in terms of $R$-matrices $\bar{R}_{ij}$ (\ref{q04})-(\ref{q05}), while the $R$-matrices in the identities (\ref{a20}) are those normalized as in (\ref{q03}).

Next, we prove
\begin{theor}\label{th2}
  The elliptic $R$-matrix (\ref{BB}) satisfies the identities (\ref{a20}).
\end{theor}
For the identity with $k=1$ a more general statement is proved. Namely, for $k=1$ the identity is proved for any $R$-matrix satisfying the quantum Yang-Baxter equation (\ref{QYB}), the associative Yang-Baxter equation (\ref{AYBE}) and the unitarity property (\ref{q03}). These conditions are fulfilled for not only elliptic $R$-matrix but also for a wide class of its trigonometric and rational degenerations \cite{Chered2,AHZ,Pol2,T,Sm,KZ19}. The identities for $k>1$ are proved in a different way. The proof is by induction in $k$. It is shown that vanishing of residues at poles in the $k$-th identity is equivalent to
($k-1$)-th identity. In this respect it is similar to the original proof from \cite{Ruij}, but technically it is much more complicated in the $R$-matrix case.

The proof of Theorem \ref{th2} is not reduced to commutativity of
qKZ-type difference operators (or qKZ connections) \cite{FrR} since such operators commute
due to the quantum Yang-Baxter equation (\ref{QYB}) only. The latter equation becomes trivial in the scalar case but not the identities, which turn in the scalar case to the original form deduced in \cite{Ruij}.

\paragraph{Remark.}  In \cite{Ruij} the operators (\ref{Dscalar}) were introduced in a slightly different form (with square roots), which is related to (\ref{Dscalar}) by conjugating all the operators with a certain function (see e.g. \cite{KH}). The operators (\ref{q10}) can be also transformed to the Ruijsenaars' like form:
$$
  \displaystyle{
    {\mathcal D}'_k=\sum\limits_{1\leq i_1<...<i_k\leq N}\left(\!\prod\limits^{N}_{\hbox{\tiny{$ \begin{array}{c}{ j=1 }\\{ j\!\neq\! i_1...i_{k-1} } \end{array}$}}}\!\phi(z_j-z_{i_1})\cdots\phi(z_j-z_{i_k})\right)^{\frac{1}{2}}\left(
   \overleftarrow{\prod\limits_{j_1=1}^{i_1-1}} \bar{R}_{j_1 i_1}
    \ldots\
 \overleftarrow{\prod\limits^{i_k-1}_{\hbox{\tiny{$ \begin{array}{c}{ j_k=1 }\\{ j_k\!\neq\! i_1...i_{k-1} } \end{array}$}}}} \bar{R}_{j_k i_k}
  \right)\times
   }
  $$
  \beq\label{q20}
     \displaystyle{
       \times p_{i_1}\cdots p_{i_k}\times\left(
   \overrightarrow{\prod\limits^{i_k-1}_{\hbox{\tiny{$ \begin{array}{c}{ j_k\!=\!1 }\\{ j_k\!\neq\! i_{1}...i_{k-1}} \end{array}$}}}}\bar{R}_{i_k j_k}
   \ldots\
  \overrightarrow{\prod\limits^{i_{1}-1}_ {j_1=1}} \bar{R}_{i_{1} j_{1}}\right)
 \left(\!\prod\limits^{N}_{\hbox{\tiny{$ \begin{array}{c}{ j=1 }\\{ j\!\neq\! i_1...i_{k-1} } \end{array}$}}}\!\phi(z_{i_1}-z_j)\cdots\phi(z_{i_k}-z_{j})\right)^{\frac{1}{2}}\,.
 }
\eq
All the above statements are valid for these operators as well.

\paragraph{The paper is organized as follows.} In Section \ref{sect2}
we introduce convenient notations for $R$-matrix products. Then the quantum Yang-Baxter equation (\ref{QYB2}) together with the property (\ref{QYB3}) are transformed into the statements (\ref{l21}) and (\ref{l22}) of Lemma \ref{lemXYZ}. Next, following ideas from \cite{Uglov} and \cite{LPS} we derive a set of the matrix valued difference operators ${\mathcal D}_k$ (\ref{q10}). Examples for $N=2,3,4$
are given in detail in the subsection \ref{sect23}. In Section \ref{sect3} we begin with recalling the derivation of identities in the scalar case. It is a slightly modified version of the original one \cite{Ruij}, which is adopted for $R$-matrix generalization. The $R$-matrix identities are deduced in Theorem \ref{th3}, which is a reformulation of the Theorem \ref{th1}. It is important to mention that
all the functions $\phi$ (entering the definition of ${\mathcal D}_k$) can be
included into $R$-matrix normalization (\ref{q04}) in the identities (\ref{a20})-(\ref{a22}).
A few explicit examples are given in the end of the Section. Section \ref{sect4} is devoted to the proof of Theorem \ref{th2}, i.e. to the proof of $R$-matrix identities. We start with identities for $k=1$ and prove them using the higher order  analogue (\ref{a4}) for the associative Yang-Baxter equation (\ref{AYBE}) likewise
the elliptic function identity (\ref{x1}) is a higher order analogue (in $\phi$-functions) of (\ref{Fay}). Then we proceed to the case $k>1$. Here
the proof is by induction in $k$. Main idea is based on the analysis of poles and quasi-periodic properties. In Section \ref{sect5} we briefly discuss some limiting cases including the limit to differential operators, trigonometric and classical limits. As a by-product of our results we also propose elliptic generalization for q-deformed Haldane spin chain. In the Appendix A some basic notations and properties of elliptic functions are given including the definition of ${\rm GL}_M$ elliptic $R$-matrix (\ref{BB}). Appendices B, C and D contain proofs of technical statements underlying derivation and proof of the $R$-matrix identities.

\section{Construction of spin operators}\label{sect2}
\setcounter{equation}{0}
In \cite{Ruij} the following notations were used.
Let $I,J$ be disjoint subsets of $\{1,\dots,N\}$. Let
\beq\label{IJ}
(I,J)=\prod\limits_{\substack{i\in I\\ j\in J}}\phi(z_{i}-z_j)
\eq
and
\beq\label{pI}
\mathbf{p}_I=\prod_{i\in I}p_{i}\,,
\eq
where $p_i$ was defined in (\ref{p_i}).
Then the operators $D_k$ (\ref{Dscalar}) take the following form:
\beq\label{D+}
D_k=\sum\limits_{I:\substack{|I|=k}}(I^c,I)\,\mathbf{p}_{I},\qquad k=1,\dots,N\,,
\eq
where $I^c$ means the complement of a set $I$ in $\{1,\dots,N\}$, and $|I|$ is the number of elements in $I$.

These notations are used for derivation and proof of the identities underlying commutativity of $D_k$. We
discuss it in the next section. Below we describe the $R$-matrix products notations and formulate a set of properties, which will be very helpful in derivation and proof of $R$-matrix identities.

%
\subsection{Convenient notations for R-matrix products}\label{sect21}
For any pair $I,J$ of disjoint subsets in $\{1,\dots,N\}$ define the product
\beq\label{RIJsh1}
\mathcal{R}_{I,J}=\prod_{i\in I, j\in J, i<j} R_{ij}\left({z_i}-{z_j}\right)\,,
\eq
where the ordering of $R$-matrices is as follows:
\begin{equation}\label{RIJ1}
\mathcal{R}_{I,J}=\overleftarrow{\prod\limits_{\substack{i_1\in I\\i_1<j_1}}} R_{i_1,j_1}\left({z_{i_1}}-{z_{j_1}}\right) \overleftarrow{\prod\limits_{\substack{i_2\in I\\i_2<j_2}}} R_{i_2,j_2}\left({z_{i_2}}-{z_{j_2}}\right)\dots\overleftarrow{\prod\limits_{\substack{i_k\in I\\i_k<j_k}}}R_{i_k,j_k}\left({z_{i_k}}-{z_{j_k}}\right)\,.
\end{equation}
Here $J=\{j_1,j_2,\dots,j_k\}$ and the elements $j_m$ are in increasing order $j_1<j_2<\dots<j_k$.
Let also $I=\{i_1,i_2,\dots,i_l\}$ and $i_1<i_2<\dots<i_l$. In what follows we assume that the above ordering of
indices in $I$ and $J$ is fixed. By moving
$R$-matrices with the help of (\ref{QYB3}) the definition (\ref{RIJ1}) is equivalently rewritten as
\begin{equation}\label{RIJ2}
\mathcal{R}_{I,J}=\overrightarrow{\prod\limits_{\substack{j_l\in J\\j_l>i_l}}}R_{i_l,j_l}\left({z_{i_l}}-{z_{j_l}}\right) \overrightarrow{\prod\limits_{\substack{j_{l-1}\in J\\ j_{l-1}>i_{l-1}}}} R_{i_{l-1},j_{l-1}}\left({z_{i_{l-1}}}-{z_{j_{l-1}}}\right)\dots\overrightarrow{\prod\limits_{\substack{j_1\in J\\j_1>i_1}}} R_{i_1,j_1}\left({z_{i_1}}-{z_{j_1}}\right)\,.
\end{equation}
 For example, consider $I=\{a_1,a_2\}$, $J=\{b_1,b_2,b_3\}$ and $1\leq a_1<a_2<b_1<b_2<b_3\leq N$. Then (\ref{RIJ1})
 reads as $\mathcal{R}_{I,J}=R_{a_2b_1}R_{a_1b_1}\cdot R_{a_2b_2}R_{a_1b_2}\cdot R_{a_2b_3}R_{a_1b_3}$. The equivalent expression (\ref{RIJ2}) takes the following form:
  $\mathcal{R}_{I,J}=R_{a_2b_1}R_{a_2b_2}R_{a_2b_3}\cdot R_{a_1b_1}R_{a_1b_2}R_{a_1b_3}$.
 By definition we assume $\mathcal{R}_{I,J}={\rm Id}$ if there are no pairs of indices satisfying $i<j$ in $I$ and $J$.
 In the above example $\mathcal{R}_{I,J}={\rm Id}$ for $J=\{a_1,a_2\}$ and $I=\{b_1,b_2,b_3\}$.

 Similarly, define the product
\beq\label{RIJsh2}
\mathcal{R}_{I,J}'=\prod_{i\in I, j\in J, i>j} R_{ij}\left({z_i}-{z_j}\right)\,,
\eq
with the following ordering:
\begin{equation}\label{RIJ'1}
\mathcal{R}_{I,J}'=\overrightarrow{\prod\limits_{\substack{j_l\in J\\j_l<i_l}}}R_{i_l,j_l}(z_{i_l}-z_{j_l})\overrightarrow{\prod\limits_{\substack{j_{l-1}\in J\\j_{l-1}<i_{l-1}}}} R_{i_{l-1},j_{l-1}}(z_{i_{l-1}}-z_{j_{l-1}})\dots \overrightarrow{\prod\limits_{\substack{j_1\in J\\j_1<i_1}}} R_{i_1,j_1}(z_{i_1}-z_{j_1})\,.
\end{equation}
Again, we may rewrite it equivalently as
\begin{equation}\label{RIJ'2}
\mathcal{R}_{I,J}'=\overleftarrow{\prod\limits_{\substack{i_1\in I\\i_1>j_1}}} R_{i_1,j_1}\left({z_{i_1}}-{z_{j_1}}\right) \overleftarrow{\prod\limits_{\substack{i_2\in I\\i_2>j_2}}} R_{i_2,j_2}\left({z_{i_2}}-{z_{j_2}}\right)\dots\overleftarrow{\prod\limits_{\substack{i_k\in I\\i_k>j_k}}} R_{i_k,j_k}\left({z_{i_k}}-{z_{j_k}}\right)\,.
\end{equation}
The product of $\mathcal{R}_{I,J}'$ and $\mathcal{R}_{I,J}$ provides $R$-matrix analogue for the notation $(I,J)$ (\ref{IJ}) used in the scalar case.
It can be easily verified that
\begin{equation}\label{q210}
 \mathcal{R}_{I,J}'\mathcal{R}_{I,J}=\overleftarrow{\prod_{i_1\in I}} R_{i_1,j_1}\left({z_{i_1}}-{z_{j_1}}\right) \overleftarrow{\prod_{i_2\in I}} R_{i_2,j_2}\left({z_{i_2}}-{z_{j_2}}\right)\dots\overleftarrow{\prod_{i_k\in I}}R_{i_k,j_k}\left({z_{i_k}}-{z_{j_k}}\right)
\end{equation}
and
\begin{equation}\label{q211}
\mathcal{R}_{I,J}'\mathcal{R}_{I,J}=\overrightarrow{\prod_{j_l\in J}} R_{i_l,j_l}(z_{i_l}-z_{j_l})\overrightarrow{\prod_{j_{l-1}\in J}} R_{i_{l-1},j_{l-1}}(z_{i_{l-1}}-z_{j_{l-1}})\dots \overrightarrow{\prod_{j_1\in J}} R_{i_1,j_1}(z_{i_1}-z_{j_1})\,.
\end{equation}

Up till now we did not use the Yang-Baxter equation (\ref{QYB2}). Let us formulate it for the products (\ref{RIJ1}) and
(\ref{RIJ'1}) since it plays a key role in deriving and proving $R$-matrix identities.
\begin{lemma}\label{lemXYZ}
 Let $R(u)$ be an $R$-matrix satisfying the quantum Yang-Baxter equation (\ref{QYB}) and $\mathcal{R}_{A,B}$ and $\mathcal{R}'_{A,B}$ are the corresponding products of $R$-matrices defined by  (\ref{RIJ1}) and (\ref{RIJ'1}) respectively. For any disjoint subsets $A,B,C$ of $\{1,2,\dots,N\}$ the following identities hold true:
\begin{equation}\label{l21}
 \mathcal{R}_{C,A\cup B}\mathcal{R}_{B,A}=\mathcal{R}_{B\cup C,A}\mathcal{R}_{C,B}\,,
\end{equation}
\begin{equation}\label{l22}
 \mathcal{R}'_{A,B}\mathcal{R}'_{A\cup B,C}=\mathcal{R}'_{B,C}\mathcal{R}'_{A,B\cup C}\,.
\end{equation}
\end{lemma}
A  sketch of the proof is given in the Appendix.

Notice that we did not use the unitarity property in the above definitions and properties.
For a unitary $R$-matrix satisfying (\ref{q03}) we have (in addition to all the above mentioned statements):
\begin{equation}\label{un01}
 \mathcal{R}_{I,J}\mathcal{R}_{J,I}'={\rm Id}\prod\limits_{\substack{i\in I,j\in J\\i<j}}\phi(h,z_i-z_j)\phi(h,z_j-z_i)
\end{equation}
and
\begin{equation}\label{un02}
 \mathcal{R}'_{I,J}\mathcal{R}_{J,I}={\rm Id}\prod\limits_{\substack{i\in I,j\in J\\i>j}}\phi(h,z_i-z_j)\phi(h,z_j-z_i)\,.
\end{equation}

All the same notations are used for the normalized unitary $R$-matrix $\bar{R}^h(u)$ satisfying (\ref{q05}).  In this case we have
\begin{equation}\label{un03}
 \bar{\mathcal{R}}_{I,J}\bar{\mathcal{R}}_{J,I}'=\bar{\mathcal{R}}_{J,I}'\bar{\mathcal{R}}_{I,J}={\rm Id}\,.
\end{equation}


\subsection{Derivation of spin operators}\label{sect22}

Using notations (\ref{IJ}), (\ref{pI}) and (\ref{RIJ1}), (\ref{RIJ'1}) introduce the following set of $\MatM^{\otimes N}$-valued difference operators:
\begin{equation}
\label{Dspin2}
\mathcal{D}_k=
\sum_{|I|=k}(I^c,I) \cdot\bar{\mathcal{R}}_{I^c,I}\cdot \mathbf{p}_{I}\cdot \bar{\mathcal{R}}_{I,I^c}'\,,
\end{equation}
where $R$-matrices with bars are those normalized as in (\ref{un03}). Explicit form for these operators
is given in (\ref{q10}). We also write down a few examples in the end of the section.

The purpose of this paragraph is to present a derivation of (\ref{Dspin2}). At the same time we emphasize that the construction recipe does not guarantee commutativity, which is proved in the next sections.
Our derivation of (\ref{Dspin2}) is based on the symmetry with respect to transpositions
\begin{equation}\label{transpqKZ}
s_{i,i+1} P_{i,i+1}\bar{R}_{i,i+1}\in{\rm End}(\mH),
\end{equation}
where $\mH=(\mC^M)^{\otimes N}$,  $s_{i,i+1}$  are operators permuting coordinates $z_i$ and $z_{i+1}$,  and $P_{i,i+1}$ are the matrix permutation operators (\ref{P12}) permuting $i$-th and $(i+1)$-th tensor components of $\mH$.
The Macdonald operators (\ref{Dscalar}), (\ref{Macd}) are symmetric with respect to permutation of coordinates and preserve the space of symmetric polynomials. Replacing the usual symmetric group with its representation, we get the matrix-valued operators (\ref{Dspin2}) which are symmetric with respect to (\ref{transpqKZ}).
Similar approach  was suggested by Uglov \cite{Uglov} and Lamers-Pasquier-Serban \cite{LPS} for trigonometric (XXZ) ${\rm GL}_2$ $R$-matrix. Similar constructions were suggested in \cite{Chered3,Jimbo,HaSe}.  The idea was to consider the $q$-bosonic Fock space  which in trigonometric case is a subspace of $\mH$ highlighted by the ``local condition'' \cite{FSmirnov} from the qKZ equation \cite{FrR}:
\begin{equation}\label{localqKZ}
s_{i,i+1} P_{i,i+1}\bar{R}_{i,i+1}={\rm Id}.
\end{equation}
However,
the commutativity of spin operators  in \cite{Uglov,LPS}  was based on the properties of q-bosonic Fock space, which is constructed as representations of affine Hecke algebra. This proof is appropriate in the trigonometric case only. In our work we deal with a more general elliptic case and prove the commutativity of operators in a different way.

Consider symmetric group $S_N$ generated by  relations:
\beq\label{braid}
\sigma_{i-1}\sigma_{i}\sigma_{i-1}=\sigma_{i}\sigma_{i-1}\sigma_{i}\,,
\eq
\beq \label{braid2}
\sigma_{i}\sigma_{j}=\sigma_{j}\sigma_{i} \qquad \text{for } j\neq i\pm 1
\eq
and
\beq\label{braid3}
(\sigma_{i})^2=1\,.
\eq
Obviously, it has representation $\sigma_{i}=s_{i,i+1}$, where $s_{i,i+1}$, $i=1,...,N-1$ are permutations of
variables $z_1,...,z_N$:
\beq\label{p0}
s_{i,i+1}f(z_1,...,z_i,z_{i+1},...,z_N)=f(z_1,...,z_{i+1},z_i,...,z_N)\,.
\eq
Here $f$ is any function (for which the action of $D_k$ operators is well-defined). Denote by $s_\omega$
the permutation operator representing $w\in S_N$. For example, for the cycle $(12\dots j)$ we have
\beq\label{cycle0}
\begin{array}{c}
s_{(12\dots j)}=s_{12}s_{23}\dots s_{j-1,j}\,.
\end{array}
\eq
For $i_1<i_2<\dots<i_k$ denote by $\{i_1,i_2,\dots,i_k\}\in S_N$ the (shortest) permutation $i_m\rightarrow m$ for all $1\le m\le k$. It can be presented as a product of cycles:
\beq\label{permut}
\{i_1,i_2,\dots,i_k\}=(k,\dots,i_k)(k-1,\dots,i_{k-1})\dots(2,\dots,i_2)(1,\dots,i_1)\,.
\eq
The Macdonald-Ruijsenaars operators (\ref{Dscalar}) or (\ref{D+}) are symmetric with respect to the action of permutation group (\ref{p0}). Therefore, each of these operators can be represented as a sum over certain permutations acting on some ''first'' term:
\beq\label{Dspin0}
D_k=\sum_{i_1<i_2<\dots<i_k}s_{\{i_1,i_2,\dots,i_k\}}^{-1}(I_0^c,I_0)\,\mathbf{p}_{I_0}\,s_{\{i_1,i_2,\dots,i_k\}}\,,
\eq
where we denote by  $I_0=\{1,2,\dots,k\}$ the subset in $k$ elements and its complement $I_0^c=\{k+1,\dots N\}$ and use the notation (\ref{IJ}). The sum in (\ref{Dspin0}) is over all (ordered) $k$-element subsets of $\{1,\dots,N\}$.

Another well known representation of the (braid) relations (\ref{braid})-(\ref{braid2}) is given by
$\sigma_{i}=R^\hbar_{i,i+1}(z_{i}-z_{i+1})P_{i,i+1}\in{\rm End}(\mH)$, $\mH=(\mC^M)^{\otimes N}$, where $P_{ij}$ are permutation matrix-valued operators (\ref{P12})-(\ref{P12-1}). In this case (\ref{braid})-(\ref{braid2}) are equivalent to the Yang-Baxter equations (\ref{QYB2})-(\ref{QYB3}). If the $R$-matrix entering representation is unitary with normalization (\ref{q05}), then the involution property (\ref{braid3}) holds as well.

Consider representation of (\ref{braid})-(\ref{braid3}) given by the composition of the previously discussed:
\beq\label{sigma}
\sigma_{i}=\bar{R}_{i,i+1}(z_{i}-z_{i+1})P_{i,i+1}s_{i,i+1}\,.
\eq
It is easy to see that (\ref{braid})-(\ref{braid3}) are fulfilled. Similarly to (\ref{cycle0}) introduce
\beq\label{cycle}
\begin{array}{c}
\sigma_{(12\dots j)}=\sigma_{12}\sigma_{23}\dots \sigma_{j-1,j}=\bar{R}_{12}(z_1-z_2)\bar{R}_{13}(z_1-z_3)\dots \bar{R}_{1j}(z_1-z_j)P_{(12\dots j)}s_{(12\dots j)}=\\ \ \\
=P_{(12\dots j)}s_{(12\dots j)}\bar{R}_{j1}(z_j-z_1)\bar{R}_{j2}(z_j-z_2)\dots \bar{R}_{j,j-1}(z_j-z_{j-1})
\end{array}
\eq
and also define $\sigma_{\{i_1,i_2,\dots,i_k\}}$ as in (\ref{permut}).

Introduce the operators ${\mathcal D}_k$ which are matrix generalization of operators (\ref{Dspin0}) acting in ${\rm End}(\mH)$:
\beq\label{Dspin1}
 {\mathcal D}_k=
 \sum_{i_1<i_2<\dots<i_k}
 \sigma_{\{i_1,i_2,\dots,i_k\}}^{-1}(I_0^c,I_0)\,\mathbf{p}_{I_0}\,\sigma_{\{i_1,i_2,\dots,i_k\}}\,.
\eq
 Let us show that after some transformations the dependence on permutations $s_{ij}$ and $P_{ij}$ is cancelled out, and the operators (\ref{Dspin1}) take the form (\ref{Dspin2}). The latter follows from the following Lemma.
\begin{lemma}\label{lem22} The following relations hold true:
 \beq \label{b3}
 \sigma_{\{i_1,i_2,\dots,i_k\}}=P_{\{i_1,i_2,\dots,i_k\}}s_{\{i_1,i_2,\dots,i_k\}}\bar{\mathcal{R}}_{I,I^c}'\,,
 \eq
 \beq \label{b4}
\sigma_{\{i_1,i_2,\dots,i_k\}^{-1}}=\bar{\mathcal{R}}_{I^c,I}P_{\{i_1,i_2,\dots,i_k\}^{-1}}s_{\{i_1,i_2,\dots,i_k\}^{-1}}\,.
 \eq
\end{lemma}
Indeed, conjugating $(I_0,I_0^c)$ by $\sigma_{\{i_1,i_2,\dots,i_k\}}$ we get
\beq
\begin{array}{c}
\sigma_{\{i_1,i_2,\dots,i_k\}}^{-1}(I_0^c,I_0)\,\mathbf{p}_{I_0}\, \sigma_{\{i_1,i_2,\dots,i_k\}}=
\\ \ \\
=
\bar{\mathcal{R}}_{I^c,I}P_{\{i_1,i_2,\dots,i_k\}^{-1}}
s_{\{i_1,i_2,\dots,i_k\}^{-1}}(I_0^c,I_0)
\,\mathbf{p}_{I_0}\,P_{\{i_1,i_2,\dots,i_k\}}s_{\{i_1,i_2,\dots,i_k\}}
\bar{\mathcal{R}}_{I,I^c}'=
\\ \ \\
 =\bar{\mathcal{R}}_{I^c,I}(I^c,I)\,\mathbf{p}_{I}\,\bar{\mathcal{R}}_{I,I^c}'\,,
\end{array}
\eq
which establishes equality of (\ref{Dspin1}) and (\ref{Dspin2}).
Let us prove the Lemma \ref{lem22}.
\begin{proof}
 Due to (\ref{permut}) we have
 \beq\label{b11}
 \sigma_{\{i_1,i_2,\dots,i_k\}}=\sigma_{(k,\dots,i_k)}
\sigma_{(k-1,\dots,i_{k-1})}\dots \sigma_{(2,\dots,i_2)}\sigma_{(1,\dots,i_1)}
 \eq
Using the second equality of (\ref{cycle}) we rewrite each cycle in (\ref{b11}) and then move all permutations $s_w$ and $P_w$ to the left.
 Let us demonstrate this for the last two cycles of (\ref{b11}):
 \beq\label{b1}
 \begin{array}{c}
\sigma_{(2,\dots,i_2)}\sigma_{(1,\dots,i_1)}=
\\ \ \\
=P_{(2,\dots ,i_2)}s_{(2,\dots ,i_2)}\bar{R}_{i_2,2}\bar{R}_{i_2,3}
\dots\bar{R}_{i_2,i_2-1}P_{(1,\dots ,i_1)}s_{(1,\dots ,i_1)}\bar{R}_{i_1,1}\bar{R}_{i_1,2}\dots \bar{R}_{i_1,i_1-1}=
\\ \ \\
P_{(2,\dots, i_2)(1,\dots ,i_1)}s_{(2,\dots ,i_2)(1,\dots, i_1)}\bar{R}_{i_2,1}\bar{R}_{i_2,2}\dots \bar{R}_{i_2,i_1-1} \bar{R}_{i_2,i_1+1}\dots
\bar{R}_{i_2,i_2-1}\bar{R}_{i_1,1}\bar{R}_{i_1,2}\dots \bar{R}_{i_1,i_1-1}=
\\ \ \\
{=\displaystyle
 P_{(2,\dots, i_2)(1,\dots, i_1)}s_{(2,\dots, i_2)(1,\dots, i_1)}\overrightarrow{\prod\limits_{\substack{j_1\in I^c\\j_2<i_2} }}\bar{R}_{i_2,j_2}\overrightarrow{\prod\limits_{\substack{j_1\in I^c\\j_1<i_1} }}\bar{R}_{i_1,j_1}\,.
}
 \end{array}
 \eq
 Here we use that for $I=\{i_1,\dots, i_k\}$ the complement set is $I^c=\{1,\dots,i_1-1,i_1+1,\dots,i_2-1,i_2+1,\dots, i_k-1,i_k+1,\dots ,N\}$. By continuing the equality  (\ref{b11}) in the same way we get
 \beq
 \begin{array}{c}
{\displaystyle
\sigma_{\{i_1,i_2,\dots,i_k\}}=
}
\\ \ \\
{\displaystyle = P_{\{i_1,i_2,\dots,i_k\}}s_{\{i_1,i_2,\dots,i_k\}}\overrightarrow{\prod\limits_{\substack{j_k\in I^c \\j_k<i_k}}}\bar{R}_{i_k,j_k}\dots \overrightarrow{\prod\limits_{\substack{j_1\in I^c\\j_2<i_2} }}\bar{R}_{i_2,j_2}\overrightarrow{\prod\limits_{\substack{j_1\in I^c\\j_1<i_1 }}}\bar{R}_{i_1,j_1}=
}
\\ \ \\
{\displaystyle
=P_{\{i_1,i_2,\dots,i_k\}}s_{\{i_1,i_2,\dots,i_k\}}\bar{\mathcal{R}}_{I,I^c}'\,.
}
 \end{array}
 \eq
 The relation (\ref{b4}) can be proved in a similar manner by moving all permutations to the right.
\end{proof}

It is necessary to stress that the obtained operators ${\mathcal D}_k$ do not commute by construction. For example, formally one can consider the elliptic functions $\phi$ in (\ref{q10}) together with the rational $R$-matrix
 ${\bar R}_{ij}=(z_i-z_j+\hbar P_{ij})/(z_i-z_j+\hbar)$. Then such operators do not commute. However, we will see that they indeed commute for elliptic $R$-matrix.

\subsection{Examples of spin operators}\label{sect23}
Let us write down examples of the operators (\ref{q10}) or (\ref{Dspin2}) for $N=2,3,4$.

\paragraph{Example. $N=2$:}
 \beq\label{ex01}
 \begin{array}{c}
{\displaystyle
{\mathcal D}_1=\phi(z_2-z_1)p_1+\phi(z_1-z_2)\bar{R}^\hbar_{12}\left(z_1-z_2\right)p_2
\bar{R}^\hbar_{21}\left(z_2-z_1\right)\,,
}
 \end{array}
 \eq
 \beq\label{ex02}
 \begin{array}{c}
\displaystyle{
{\mathcal D}_2=p_1p_2\,.
}
 \end{array}
 \eq

\paragraph{Example. $N=3$:}
 \beq\label{ex03}
 \begin{array}{lll}
{\mathcal D}_1 &=&\phi(z_2-z_1)\phi(z_3-z_1)p_1+
\\ \ \\
& &+\phi(z_1-z_2)\phi(z_3-z_2)\bar{R}^\hbar_{12}\left(z_1-z_2\right)p_2 \bar{R}^\hbar_{21}\left(z_2-z_1\right)+
\\ \ \\
& &+\phi(z_1-z_3)\phi(z_2-z_3)\bar{R}^\hbar_{23}\left(z_2-z_3\right)\bar{R}^\hbar_{13}\left(z_1-z_3\right)p_3 \bar{R}^\hbar_{31}\left(z_3-z_1\right)\bar{R}^\hbar_{32}\left(z_3-z_2\right)\,,
 \end{array}
 \eq
 \beq\label{ex04}
 \begin{array}{lll}
{\mathcal D}_2&=&\phi(z_3-z_1)\phi(z_3-z_2)p_1p_2+
\\ \ \\
& &+\phi(z_2-z_1)\phi(z_2-z_3)\bar{R}^\hbar_{23}\left(z_2-z_3\right)p_1p_3 \bar{R}^\hbar_{32}\left(z_3-z_2\right)+
\\ \ \\
& &+\phi(z_1-z_2)\phi(z_1-z_3)\bar{R}^\hbar_{12}\left(z_1-z_2\right)\bar{R}^\hbar_{13}\left(z_1-z_3\right)p_2p_3 \bar{R}^\hbar_{31}\left(z_3-z_1\right)\bar{R}^\hbar_{21}\left(z_2-z_1\right)\,,
 \end{array}
 \eq
 and ${\mathcal D}_3=p_1p_2p_3$.

\paragraph{Example. $N=4$:}
 \beq\label{ex06}
 \begin{array}{lll}
 {\mathcal D}_1 & = &
 \displaystyle{
 \phi(z_{21})\phi(z_{31})\phi(z_{41})p_1+
 }
\\ \ \\
 & &
 \displaystyle{
 +\phi(z_{12})\phi(z_{32})\phi(z_{42})\bar{R}_{12}p_2\bar{R}_{21}+
}
\\ \ \\
 & &
 \displaystyle{
 +\phi(z_{13})\phi(z_{23})\phi(z_{43})\bar{R}_{23}\bar{R}_{13}p_3\bar{R}_{31}\bar{R}_{32}+
}
\\ \ \\
& &
\displaystyle{
+\phi(z_{14})\phi(z_{24})\phi(z_{34})\bar{R}_{34}\bar{R}_{24}\bar{R}_{14}p_4\bar{R}_{41}\bar{R}_{42}\bar{R}_{43}\,,\qquad\qquad\qquad\quad\ \
}
 \end{array}
 \eq
 \beq\label{ex07}
 \begin{array}{lll}
 {\mathcal D}_2 & = & \phi(z_{31})\phi(z_{32})\phi(z_{41})\phi(z_{42})p_1p_2+
 \\ \ \\
 & &+\phi(z_{21})\phi(z_{23})\phi(z_{41})\phi(z_{43})\bar{R}_{23}p_1p_3\bar{R}_{32}+
 \\ \ \\
 & &+\phi(z_{21})\phi(z_{24})\phi(z_{31})\phi(z_{34})\bar{R}_{34}\bar{R}_{24}p_1p_4\bar{R}_{42}\bar{R}_{43}+
 \\ \ \\
& &+\phi(z_{12})\phi(z_{13})\phi(z_{42})\phi(z_{43})\bar{R}_{12}\bar{R}_{13}p_2p_3\bar{R}_{31}\bar{R}_{21}+
\\ \ \\
& &+\phi(z_{12})\phi(z_{14})\phi(z_{32})\phi(z_{34})\bar{R}_{12}\bar{R}_{34}\bar{R}_{14}p_2p_4\bar{R}_{41}\bar{R}_{43}\bar{R}_{21}+
 \\ \ \\
& &+\phi(z_{13})\phi(z_{14})\phi(z_{23})\phi(z_{24})\bar{R}_{23}\bar{R}_{13}\bar{R}_{24}\bar{R}_{14}p_3p_4\bar{R}_{41}\bar{R}_{42}\bar{R}_{31}\bar{R}_{32}\,,
 \end{array}
 \eq
 \beq\label{ex08}
 \begin{array}{lll}
 {\mathcal D}_3& = &\phi(z_{41})\phi(z_{42})\phi(z_{43})p_1p_2p_3+
  \\ \ \\
& & +\phi(z_{31})\phi(z_{32})\phi(z_{34})\bar{R}_{34}p_1p_2p_4\bar{R}_{43}
  \\ \ \\
& & +\phi(z_{21})\phi(z_{23})\phi(z_{24})\bar{R}_{23}\bar{R}_{24}p_1p_3p_4\bar{R}_{42}\bar{R}_{32}+
  \\ \ \\
& & +\phi(z_{12})\phi(z_{13})\phi(z_{14})\bar{R}_{12}\bar{R}_{13}\bar{R}_{14}p_2p_3p_4\bar{R}_{41}\bar{R}_{31}\bar{R}_{21}\,\qquad\qquad\
 \end{array}
 \eq
and ${\mathcal D}_4=p_1p_2p_3p_4$.

\section{Derivation of $R$-matrix identities}\label{sect3}
\setcounter{equation}{0}

We first recall the derivation of identities in the scalar case. It is a slightly modified version of the
original derivation by S.N.M. Ruijsenaars \cite{Ruij} adapted for generalization to the spin case. Then we proceed to $R$-matrix identities and finally derive (\ref{a20})-(\ref{a22}).

\subsection{Identities for scalar Macdonald-Ruijsenaars operators}\label{sect31}

Here we recall the derivation of identities in the scalar case.
%
Besides (\ref{D+}) introduce also
\beq\label{D-}
D_{-k}=\sum\limits_{\substack{|I|=k}}(I,I^c)\,\mathbf{p}_{I}^{-1},\qquad k=1,\dots,N.
\eq
It can be verified that
\beq\label{D--}
D_{k-N}=D_k D_{-N}\qquad k=1,\dots,N-1.
\eq
We use the notations $I_+$ and $I_-$ to highlight the shifts of all $z_i,\ i\in I$ by $\pm\eta$ respectively:
\beq
(I_{\pm},J)=(\mathbf{p}_I^{\mp 1} I \mathbf{p}_I^{\pm 1},J).
\eq
It is easy to check the following properties:
\beq\label{p1}
(I_+,J_+)=(I_-,J_-)=(I,J)\,,
\eq
\beq\label{p2}
(I,J_+)=(I_-,J) \qquad (I_+,J)=(I,J_-)\,.
\eq
They hold true since the expression (\ref{IJ}) depends on differences of coordinates $(z_i-z_j)$ only.
\begin{predl}\label{propscalar}
\cite{Ruij} The commutativity
 \beq\label{com}
 [D_k,D_l]=0\qquad \forall k,l=1,\dots ,N\
 \eq
  holds if and only if
  \beq\label{Idsc}
\sum\limits_{\substack{|I|=k}}\left((I^c,I)(I_-,I^c)-(I,I^c)(I^c_-,I)\right)=0 \qquad \forall k\in\{1,\dots,N\}\quad \forall N.
  \eq
\end{predl}
The relation (\ref{Idsc}) is given explicitly in (\ref{IdenRuij}). It can be viewed as a functional equation for the function $\phi(z)$.
\begin{proof}
Since $D_{\pm N}$ commutes with all $D_k$ for $k=\pm1,\dots,\pm(N-1)$ and due to the relation (\ref{D--}), it is enough to verify  commutativity (\ref{com})  for $D_k$ and $D_{-l }$,  $ k,l=1,\dots ,N$.
Using (\ref{D+}) and (\ref{D-}) we obtain
\beq\label{a797}
[D_k,D_{-l}]=\sum\limits_{\substack{|I|=k\\ |J|=l}}\left((I^c,I)\mathbf{p}_{I}(J,J^c)\mathbf{p}_{J}^{-1}-(J,J^c)\mathbf{p}_{J}^{-1}(I^c,I)\mathbf{p}_{I}\right)\,.
\eq
Introduce the pairwise disjoint sets
\beq
A=I\setminus J,\quad B=J\setminus I, \quad C=I\cap J, \quad D=(I\cup J)^c\,.
\eq
Then, using the properties  (\ref{p1}) and (\ref{p2}), the r.h.s. of (\ref{a797}) takes the form:
\beq\label{Dcomres}
[D_k,D_{-l}]=\sum\limits_{\substack{|I|=k\\ |J|=l}}(B\cup C\cup D,A)(B,C\cup D)(B,A_-)\left((D,C)(C_-,D)-(C,D)(D_-,C)\right)\mathbf{p}_A\mathbf{p}_B^{-1}\,.
\eq
The expression $(B\cup C\cup D,A)(B,C\cup D)(B,A_-)$ depends on sets $A$ and $B$ only. Thus, the coefficient at $\mathbf{p}_A\mathbf{p}_{B}^{-1}$ vanishes if and only if
\beq\label{CD}
\sum\limits_{\substack{|C|=k-|A|\\ |D|=N-k-|B|\\C\cup D=(A\cup B)^c}}\left((D,C)(C_-,D)-(C,D)(D_-,C)\right)=0\,,
\eq
where the sum is over all disjoint sets $C$ and $D$ such that $C\cup D=(A\cup B)^c$ and $|C|=k-|A|,\ |D|=N-k-|B|$. The identity (\ref{CD}) is equivalent to (\ref{Idsc}).
\end{proof}

It was argued in \cite{Ruij} that the identities (\ref{Idsc}) considered as a functional equations for the function $\phi$ are reduced
to a single equation,
which provides (among meromorphic functions on elliptic curve) solution given by the elliptic Kronecker function (\ref{q01}) only (up to some normalization factors). Trigonometric and rational solutions are obtained by degeneration as given in
(\ref{a081}).

\subsection{$R$-matrix identities for spin Macdonald-Ruijsenaars operators}\label{sect32}
In this Section we show that the commutativity of spin Macdonald-Ruijsenaars operators is equivalent to a set of $R$-matrix identities. Let us reformulate the Theorem \ref{th1} using notations
(\ref{RIJ1}), (\ref{RIJ'1}).
\begin{theor}\label{th3}
The operators $\mathcal{D}_k$ (\ref{Dspin2}),(\ref{q10}) commute with each other iff the following set of identities holds for any $k=1,2,\dots m$ and any $m\le N$:
\begin{equation}\label{relRI}
\sum_{|I|=k}\left(\mathcal{R}_{I^c,I}\cdot\mathcal{R}_{I_-,I^c}'\cdot\mathcal{R}_{I_-,I^c}\cdot\mathcal{R}_{I^c,I}'-\mathcal{R}_{I,I^c}\cdot\mathcal{R}_{I^c_-,I}'\cdot\mathcal{R}_{I^c_-,I}\cdot\mathcal{R}_{I,I^c}'\right)=0.
\end{equation}
\end{theor}
The identities (\ref{relRI}) are written explicitly in (\ref{a20})-(\ref{a22}). Before proving Theorem \ref{th3} we formulate the following technical lemma.
\begin{lemma}\label{lemma1} Let $I$ and $J$ be subsets of $\{1,2,3,\dots,N\}$. The following identity is valid:
\begin{equation}\label{lem1}
\begin{array}{c}
 \bar{\mathcal{R}}_{I^c,I}\cdot\mathbf{p}_{I}\cdot\bar{\mathcal{R}}_{I,I^c}'\cdot\bar{\mathcal{R}}_{J,J^c}\cdot\mathbf{p}_J^{-1}\cdot\bar{\mathcal{R}}'_{J^c,J}=
 \\ \ \\
 =\bar{\mathcal{R}}_{B\cup C\cup D,A}\cdot\bar{\mathcal{R}}_{B,C\cup D}\cdot\mathbf{p}_{A}\cdot\left(\bar{\mathcal{R}}_{D,C}\cdot\bar{\mathcal{R}}_{C_-,D}'\cdot\bar{\mathcal{R}}_{C_-,D}\cdot\bar{\mathcal{R}}'_{D,C}\right)\cdot \mathbf{p}_{B}^{-1}\cdot\bar{\mathcal{R}}_{C\cup D,B}'\cdot\bar{\mathcal{R}}_{A,B\cup C\cup D}'\,,
\end{array}
\end{equation}
where in the r.h.s. the notations $C=I\cap J$, $A=I\setminus C$, $B=J\setminus C$, $D=(I\cup J)^c$ are used.
\end{lemma}
\begin{proof}
The proof is based on the Yang-Baxter equation for the $R$-matrices, Lemma \ref{lemXYZ}, and on the unitarity property (\ref{un03}).
First, let us rewrite the l.h.s of (\ref{lem1}) in terms of the subsets $A,B,C,D$:
\begin{equation}\label{q301}
\begin{array}{c}
 \bar{\mathcal{R}}_{I^c,I}\cdot\mathbf{p}_{I}\cdot\bar{\mathcal{R}}_{I,I^c}'\cdot\bar{\mathcal{R}}_{J,J^c}\cdot\mathbf{p}_J^{-1}\cdot\bar{\mathcal{R}}'_{J^c,J}=
 \\ \ \\
 =\bar{\mathcal{R}}_{B\cup D,A\cup C}\cdot\mathbf{p}_{A}\cdot\mathbf{p}_{C}\cdot\bar{\mathcal{R}}'_{A\cup C,B\cup D}\cdot\bar{\mathcal{R}}_{B\cup C,A\cup D}\cdot\mathbf{p}_{B}^{-1}\cdot\mathbf{p}_{C}^{-1}\cdot\bar{\mathcal{R}}'_{A\cup D,B\cup C}\,.
 \end{array}
\end{equation}
The next step is to transform the middle $R$-matrix items in the r.h.s. of (\ref{q301}).
Consider $\bar{\mathcal{R}}_{B\cup C,A\cup D}$ and insert an identity operator  ${\rm Id}=\bar{\mathcal{R}}_{D,A}\cdot\bar{\mathcal{R}}_{C,B}\cdot\bar{\mathcal{R}}'_{A,D}\cdot\bar{\mathcal{R}}'_{B,C}$ to the right. Then,  using (\ref{l21}) transform it as
\begin{equation}\label{q302}
 \begin{array}{c}
\bar{\mathcal{R}}_{B\cup C,A\cup D}=\bar{\mathcal{R}}_{B\cup C,A\cup D}\cdot\bar{\mathcal{R}}_{D,A}\cdot\bar{\mathcal{R}}_{C,B}\cdot\bar{\mathcal{R}}'_{A,D}\cdot\bar{\mathcal{R}}'_{B,C}=
\\ \ \\
=\bar{\mathcal{R}}_{B\cup C\cup D,A}\cdot\bar{\mathcal{R}}_{B\cup C,D}\cdot\bar{\mathcal{R}}_{C,B}\cdot\bar{\mathcal{R}}'_{A,D}\cdot\bar{\mathcal{R}}'_{B,C}
=\bar{\mathcal{R}}_{B\cup C\cup D,A}\cdot\bar{\mathcal{R}}_{B,D\cup C}\cdot\bar{\mathcal{R}}_{C,D}\cdot\bar{\mathcal{R}}'_{A,D}\cdot\bar{\mathcal{R}}'_{B,C}\,.
\end{array}
\end{equation}
In the same way using  (\ref{l22}) one obtains
\begin{equation}\label{q303}
 \begin{array}{c}
\bar{\mathcal{R}}'_{A\cup C,B\cup D}=\bar{\mathcal{R}}_{C,A}\cdot\bar{\mathcal{R}}_{B,D}\cdot\bar{\mathcal{R}}'_{D,B}\cdot\bar{\mathcal{R}}'_{A,C}\cdot\bar{\mathcal{R}}'_{A\cup C,B\cup D}=
\\ \ \\
=\bar{\mathcal{R}}_{C,A}\cdot\bar{\mathcal{R}}_{B,D}\cdot\bar{\mathcal{R}}'_{D,B}\cdot\bar{\mathcal{R}}'_{C,B\cup D}\cdot\bar{\mathcal{R}}'_{A,C\cup B\cup D}
=\bar{\mathcal{R}}_{C,A}\cdot\bar{\mathcal{R}}_{B,D}\cdot\bar{\mathcal{R}}'_{C,D}\cdot\bar{\mathcal{R}}'_{C\cup D,B}\cdot\bar{\mathcal{R}}'_{A,C\cup B\cup D}\,.
\end{array}
\end{equation}
Multiply the r.h.s. of (\ref{q303}) and the r.h.s. of (\ref{q302}). The factors  $\bar{\mathcal{R}}'_{C\cup D,B}\cdot\bar{\mathcal{R}}'_{A,C\cup B\cup D}$ and $\bar{\mathcal{R}}_{B\cup C\cup D,A}\cdot\bar{\mathcal{R}}_{B,D\cup C}$ are reduced due to the unitarity property (\ref{un03}), so that the answer is
\begin{equation}\label{q304}
 \begin{array}{c}
\bar{\mathcal{R}}'_{A\cup C,B\cup D}\cdot\bar{\mathcal{R}}_{B\cup C,A\cup D}
=\bar{\mathcal{R}}_{C,A}\cdot\bar{\mathcal{R}}_{B,D}\cdot\bar{\mathcal{R}}'_{C,D}\cdot\bar{\mathcal{R}}_{C,D}\cdot\bar{\mathcal{R}}'_{A,D}\cdot\bar{\mathcal{R}}'_{B,C}\,.
\end{array}
\end{equation}
Plugging (\ref{q304}) into the r.h.s. of (\ref{q301}) one obtains:
\beq\label{q305}
\begin{array}{c}
\bar{\mathcal{R}}_{B\cup D,A\cup C}\cdot\mathbf{p}_{A}\cdot\mathbf{p}_{C}\cdot\bar{\mathcal{R}}'_{A\cup C,B\cup D}\cdot\bar{\mathcal{R}}_{B\cup C,A\cup D}\cdot\mathbf{p}_{B}^{-1}\cdot\mathbf{p}_{C}^{-1}\cdot\bar{\mathcal{R}}'_{A\cup D,B\cup C}=
\\ \ \\
=\bar{\mathcal{R}}_{B\cup D,A\cup C}\cdot\mathbf{p}_{A}\cdot\mathbf{p}_{C}\cdot\bar{\mathcal{R}}_{C,A}\cdot\bar{\mathcal{R}}_{B,D}\cdot\bar{\mathcal{R}}'_{C,D}\cdot\bar{\mathcal{R}}_{C,D}\cdot\bar{\mathcal{R}}'_{A,D}\cdot\bar{\mathcal{R}}'_{B,C}\cdot\mathbf{p}_{B}^{-1}\cdot\mathbf{p}_{C}^{-1}\cdot\bar{\mathcal{R}}'_{A\cup D,B\cup C}=
\\ \ \\
=\bar{\mathcal{R}}_{B\cup D,A\cup C}\cdot\bar{\mathcal{R}}_{C,A}\cdot\bar{\mathcal{R}}_{B,D}\cdot\mathbf{p}_{A}\cdot\mathbf{p}_{C}\cdot\bar{\mathcal{R}}'_{C,D}\cdot\bar{\mathcal{R}}_{C,D}\cdot\mathbf{p}_{B}^{-1}\cdot\mathbf{p}_{C}^{-1}\cdot\bar{\mathcal{R}}'_{A,D}\cdot\bar{\mathcal{R}}'_{B,C}\cdot\bar{\mathcal{R}}'_{A\cup D,B\cup C}=
\\ \ \\
=\bar{\mathcal{R}}_{B\cup D,A\cup C}\cdot\bar{\mathcal{R}}_{C,A}\cdot\bar{\mathcal{R}}_{B,D}\cdot\mathbf{p}_{A}\cdot\bar{\mathcal{R}}'_{C_-,D}\cdot\bar{\mathcal{R}}_{C_-,D}\cdot\mathbf{p}_{B}^{-1}\cdot\bar{\mathcal{R}}'_{A,D}\cdot\bar{\mathcal{R}}'_{B,C}\cdot\bar{\mathcal{R}}'_{A\cup D,B\cup C}\,.
\end{array}
\eq
In the second equality of (\ref{q305}) we use that $\mathbf{p}_{A}\cdot\mathbf{p}_{C}$ commutes with $\bar{\mathcal{R}}_{C,A}$ and $\bar{\mathcal{R}}_{B,D}$ since the first one depends on differences $z_a-z_c$, where $a\in A\, ,c\in C$, and the second does not depend on subsets $A$ and $C$. Similarly, $\mathbf{p}_{B}^{-1}\cdot\mathbf{p}_{C}^{-1}$ commutes with $\bar{\mathcal{R}}'_{A,D}$ and $\bar{\mathcal{R}}'_{B,C}$. In the last equality we move $\mathbf{p}_{C}$ to the right to reduce it with $\mathbf{p}_{C}^{-1}$.

Finally, to prove (\ref{lem1}) one should substitute the following consequences of Lemma \ref{lemXYZ}:
\beq\label{q306}
\bar{\mathcal{R}}_{B\cup D,A\cup C}\cdot\bar{\mathcal{R}}_{C,A}\cdot\bar{\mathcal{R}}_{B,D}=\bar{\mathcal{R}}_{B\cup C\cup D,A}\cdot\bar{\mathcal{R}}_{B\cup D,C}\cdot\bar{\mathcal{R}}_{B,D}=\bar{\mathcal{R}}_{B\cup C\cup D,A}\cdot\bar{\mathcal{R}}_{B,C\cup D}\cdot\bar{\mathcal{R}}_{D,C}
\eq
\beq\label{q307}
\bar{\mathcal{R}}'_{B,C}\cdot\bar{\mathcal{R}}'_{A,D}\cdot\bar{\mathcal{R}}'_{A\cup D,B\cup C}=\bar{\mathcal{R}}'_{B,C}\cdot\bar{\mathcal{R}}'_{D,B\cup C}\cdot\bar{\mathcal{R}}'_{A,B\cup C\cup D}=\bar{\mathcal{R}}'_{D,C}\cdot\bar{\mathcal{R}}'_{C\cup D,B}\cdot\bar{\mathcal{R}}'_{A,B\cup C\cup D}
\eq
into (\ref{q305}) and use the commutativity of $\mathbf{p}_{A}$ with $\bar{\mathcal{R}}_{D,C}$ and $\mathbf{p}_{B}^{-1}$ with $\bar{\mathcal{R}}'_{D,C}$.
\end{proof}
\\
Besides (\ref{Dspin2}) introduce also:
\beq\label{q308}
\mathcal{D}_{-k}=\sum_{|I|=k}(I,I^c)\cdot \mathcal{\bar{R}}_{I,I^c}\cdot \mathbf{p}_{I}^{-1}\cdot\bar{\mathcal{R}}_{I^c,I}'\, .
\eq
\\
\textbf{Proof of Theorem \ref{th3}.}
Notice that  (\ref{relRI}) is equivalent to the identity:
\begin{equation}\label{relRIun}
\begin{array}{c}
{\displaystyle
\sum_{|I|=k}\left((I^c,I)(I_-,I^c)\cdot\mathcal{\bar{R}}_{I^c,I}\cdot\mathcal{\bar{R}}_{I_-,I^c}'\cdot\mathcal{\bar{R}}_{I_-,I^c}\cdot\mathcal{\bar{R}}_{I^c,I}'\right.}
\\ \ \\
\left.-(I,I^c)(I^c_-,I)\cdot\mathcal{\bar{R}}_{I,I^c}\cdot\mathcal{\bar{R}}_{I^c_-,I}'\cdot\mathcal{\bar{R}}_{I^c_-,I}\cdot\mathcal{\bar{R}}_{I,I^c}'\right)=0\,.
\end{array}
\end{equation}
Indeed, following (\ref{un03}) and the notations (\ref{IJ}), (\ref{RIJsh1}) and (\ref{RIJsh2}) we have, for example
\beq\label{q310}
\begin{array}{c}
\mathcal{R'}_{I_-,I^c}\cdot\mathcal{R}_{I_-,I^c}=
\\ \ \\
=\prod\limits_{\substack{i\in I}} \prod\limits_{\substack{j\in I^c\\j<i}}\phi(z_i-z_j-\eta)\cdot\bar{\mathcal{R'}}_{I_-,I^c}\prod\limits_{\substack{i\in I}} \prod\limits_{\substack{j\in I^c\\j>i}}\phi(z_i-z_j-\eta)\cdot\bar{\mathcal{R}}_{I_-,I^c}=
\\ \ \\
=(I_-,I^c)\cdot\bar{\mathcal{R'}}_{I_-,I^c}\cdot\bar{\mathcal{R}}_{I_-,I^c}\, .
\end{array}
\eq
The idea of the proof is the same as in Proposition \ref{propscalar}, i.e. it is sufficient to prove that
$\mathcal{D}_k$ and $\mathcal{D}_{-l }$ commute for $ k,l=1,\dots ,N$. We transform the commutator $[\mathcal{D}_k,\mathcal{D}_{-l}]$ using the pairwise disjoint sets and  consider separately the items with $R$-matrices and the scalar items of the form (\ref{IJ}). Then  the identities (\ref{relRIun}) follow from Lemma \ref{lemma1} and the result (\ref{Dcomres}) from the previous subsection.

Let us write the commutator of $\mathcal{D}_k$ and $\mathcal{D}_{-l }$ using the definitions (\ref{Dspin2}) and (\ref{q308}):
\begin{equation}\label{q311}
\begin{array}{c}
 {\displaystyle
 [\mathcal{D}_k,\mathcal{D}_{-l}]=\sum\limits_{\substack{|I|=k\\|J|=l}}\left((I^c,I)\cdot\mathcal{\bar{R}}_{I^c,I}\cdot\mathbf{p}_{I}\cdot\mathcal{\bar{R}}_{I,I^c}'(J,J^c) \cdot\mathcal{\bar{R}}_{J,J^c}\cdot\mathbf{p}_{J}^{-1}\cdot\mathcal{\bar{R}}_{J^c,J}'\right.-}
 \\ \ \\
 \left.-(J,J^c) \cdot\mathcal{\bar{R}}_{J,J^c}\cdot\mathbf{p}_{J}^{-1}\cdot\mathcal{\bar{R}}_{J^c,J}'\cdot(I^c,I)\cdot\mathcal{\bar{R}}_{I^c,I}\mathbf{p}_{I}\cdot\mathcal{\bar{R}}_{I,I^c}'\right)\,.
 \end{array}
 \end{equation}
Introduce the pairwise disjoint sets
\beq
A=I\setminus J,\quad B=J\setminus I, \quad C=I\cap J, \quad D=(I\cup J)^c\,.
\eq
Apply (\ref{lem1})  to each of $R$-matrix expressions from  (\ref{q311}):
 \begin{equation}\label{q312}
 \begin{split}
 \mathcal{\bar{R}}_{I^c,I}\cdot\mathbf{p}_{I}\cdot\mathcal{\bar{R}}_{I,I^c}'\cdot\mathcal{\bar{R}}_{J,J^c}\cdot\mathbf{p}_{J}^{-1}\cdot\mathcal{\bar{R}}_{J^c,J}'=\\
 =\mathcal{\bar{R}}_{B\cup C\cup D,A}\cdot\mathcal{\bar{R}}_{ B,C\cup D}\cdot\mathbf{p}_{A}\left(\mathcal{\bar{R}}_{D,C}\cdot\mathcal{\bar{R}}_{C_-,D}'\cdot\mathcal{\bar{R}}_{C_-,D}\cdot\mathcal{\bar{R}}'_{D,C}\right)\mathbf{p}_{B}^{-1}\cdot\mathcal{\bar{R}}_{C\cup D,B}'\cdot\mathcal{\bar{R}}_{A,B\cup C\cup D}'
 \end{split}
 \end{equation}
  \begin{equation}\label{q313}
 \begin{split}
\mathcal{\bar{R}}_{J,J^c}\cdot\mathbf{p}_{J}^{-1}\cdot\mathcal{\bar{R}}_{J^c,J}'\cdot\mathcal{\bar{R}}_{I^c,I}\cdot\mathbf{p}_{I}\cdot\mathcal{\bar{R}}_{I,I^c}'=\\
 =\mathcal{\bar{R}}_{B\cup C\cup D,A}\cdot\mathcal{\bar{R}}_{ B,C\cup D}\cdot\mathbf{p}_{B}^{-1}\left(\mathcal{\bar{R}}_{C,D}\cdot\mathcal{\bar{R}}_{D_-,C}'\cdot\mathcal{\bar{R}}_{D_-,C}\cdot\mathcal{\bar{R}}'_{C,D}\right)\mathbf{p}_{A}\cdot\mathcal{\bar{R}}_{C\cup D,B}'\cdot\mathcal{\bar{R}}_{A,B\cup C\cup D}'\,.
 \end{split}
 \end{equation}
 Taking also into account (\ref{q312}), (\ref{q313}), (\ref{Dcomres}) we rewrite the r.h.s. of (\ref{q311}) in the form
 \begin{equation}\label{q314}
 \begin{array}{c}
 [\mathcal{D}_k,\mathcal{D}_{-l}]
 =(B\cup C\cup D,A)(B,C\cup D)(B,A_-)\cdot\mathcal{\bar{R}}_{B\cup C\cup D,A}\cdot\mathcal{\bar{R}}_{ B,C\cup D}\cdot\mathbf{p}_{A}\times
 \\ \ \\
 \times \Big((D,C)(C_-,D)\cdot\mathcal{\bar{R}}_{D,C}\cdot\mathcal{\bar{R}}_{C_-,D}'\cdot\mathcal{\bar{R}}_{C_-,D}\cdot\mathcal{\bar{R}}'_{D,C}-
 \hspace{6cm}\
  \\ \ \\
 \hspace{6cm}\ -(C,D)(D_-,C)\cdot\mathcal{\bar{R}}_{C,D}\cdot\mathcal{\bar{R}}_{D_-,C}'\cdot\mathcal{\bar{R}}_{D_-,C}\cdot\mathcal{\bar{R}}'_{C,D}\Big)\times
 \\ \ \\
 \times\mathbf{p}_{B}^{-1}\cdot\mathcal{\bar{R}}_{C\cup D,B}'\cdot\mathcal{\bar{R}}_{A,B\cup C\cup D}'\,.
 \end{array}
 \end{equation}
 The first and the last line of (\ref{q314}) depends on the sets $A$ and $B$ only, and does not depend on how  $C\cup D=(A\cup B)^c$ is divided into two sets $C$ and $D$.
 Thus, the coefficient at $\mathbf{p}_A\mathbf{p}_{B}^{-1}$ vanishes if and only if
 \beq\label{q315}
 \begin{array}{ccc}
 {\displaystyle
 \sum\limits_{\substack{C,D\\ C\cup D = (A\cup B)^c}}}\!&
 \!\left((D,C)(C_-,D)\cdot\mathcal{\bar{R}}_{D,C}\cdot\mathcal{\bar{R}}_{C_-,D}'\cdot\mathcal{\bar{R}}_{C_-,D}\cdot\mathcal{\bar{R}}'_{D,C}\right.-&
\\
 &-\left.(C,D)(D_-,C)\cdot\mathcal{\bar{R}}_{C,D}\cdot\mathcal{\bar{R}}_{D_-,C}'\cdot\mathcal{\bar{R}}_{D_-,C}\cdot\mathcal{\bar{R}}'_{C,D}\right)
 &=0\,,
 \end{array}
 \eq
 where the sum is over all disjoint sets $C$ and $D$ such that $C\cup D=(A\cup B)^c$ and $|C|=k-|A|,\ |D|=N-k-|B|$. The identity (\ref{q315}) is equivalent to (\ref{relRIun}). This finishes the proof.
\hfill$\scriptstyle\blacksquare$

Let us write down several examples of identities. The non-trivial identity for $k=2$ appears beginning with $N=5$.

\paragraph{Example. $k=1$ $N=3$:}
\begin{equation}\label{rel13}
 \begin{split}
 R_{12}^h(z_{1}-z_{2})R_{13}^h(z_{1}-z_{3})R_{31}^h(z_{3}-z_{1}-\eta)R_{21}^h(z_{2}-z_{1}-\eta)\\-R_{12}^h(z_{1}-z_{2}-\eta)R_{13}^h(z_{1}-z_{3}-\eta)R_{31}^h(z_{3}-z_{1})R_{21}^h(z_{2}-z_{1})\\+
 R_{23}^h(z_{2}-z_{3})R_{32}^h(z_{3}-z_{2}-\eta)R_{12}^h(z_{1}-z_{2}-\eta)R_{21}^h(z_{2}-z_{1})\\-R_{12}^h(z_{1}-z_{2})R_{21}^h(z_{2}-z_{1}-\eta)R_{23}^h(z_{2}-z_{3}-\eta)R_{32}^h(z_{3}-z_{2})\\+
 R_{23}^h(z_{2}-z_{3}-\eta)R_{13}^h(z_{1}-z_{3}-\eta)R_{31}^h(z_{3}-z_{1})R_{32}^h(z_{3}-z_{2})\\-R_{23}^h(z_{2}-z_{3})R_{13}^h(z_{1}-z_{3})R_{31}^h(z_{3}-z_{1}-\eta)R_{32}^h(z_{3}-z_{2}-\eta)=0
 \end{split}
\end{equation}
\paragraph{Example. $k=1$ $N=4$:}
   \begin{equation}\label{rel14}
       \begin{split}
       R_{12}^h(z_{1}-z_{2}-\eta)R_{13}^h(z_{1}-z_{3}-\eta)  R_{14}^h(z_{1}-z_{4}-\eta) R_{41}^h(z_{4}-z_{1})  R_{31}^h(z_{3}-z_{1})  R_{21}^h(z_{2}-z_{1})\\-  R_{12}^h(z_{1}-z_{2})
  R_{13}^h(z_{1}-z_{3})  R_{14}^h(z_{1}-z_{4})  R_{41}^h(z_{4}-z_{1}-\eta)
   R_{31}^h(z_{3}-z_{1}-\eta)  R_{21}^h(z_{2}-z_{1}-\eta)  \\ +
      R_{12}^h(z_{1}-z_{2})
  R_{21}^h(z_{2}-z_{1}-\eta)  R_{23}^h(z_{2}-z_{3}-\eta)
  R_{24}^h(z_{2}-z_{4}-\eta)  R_{42}^h(z_{4}-z_{2})  R_{32}^h(z_{3}-z_{2})\\-
R_{23}^h(z_{2}-z_{3})   R_{24}^h(z_{2}-z_{4})  R_{42}^h(z_{4}-z_{2}-\eta)
  R_{32}^h(z_{3}-z_{2}-\eta)  R_{12}^h(z_{1}-z_{2}-\eta)
  R_{21}^h(z_{2}-z_{1}) \\+
      R_{23}^h(z_{2}-z_{3}) R_{13}^h(z_{1}-z_{3})  R_{31}^h(z_{3}-z_{1}-\eta)
  R_{32}^h(z_{3}-z_{2}-\eta)  R_{34}^h(z_{3}-z_{4}-\eta)
  R_{43}^h(z_{4}-z_{3}) \\ -
      R_{34}^h(z_{3}-z_{4})
  R_{43}^h(z_{4}-z_{3}-\eta)  R_{23}^h(z_{2}-z_{3}-\eta)
  R_{13}^h(z_{1}-z_{3}-\eta)  R_{31}^h(z_{3}-z_{1})  R_{32}^h(z_{3}-z_{2})\\+
      R_{34}^h(z_{3}-z_{4})
  R_{24}^h(z_{2}-z_{4})  R_{14}^h(z_{1}-z_{4})  R_{41}^h(z_{4}-z_{1}-\eta)
   R_{42}^h(z_{4}-z_{2}-\eta)  R_{43}^h(z_{4}-z_{3}-\eta)\\-
      R_{34}^h(z_{3}-z_{4}-\eta)
  R_{24}^h(z_{2}-z_{4}-\eta)  R_{14}^h(z_{1}-z_{4}-\eta)
  R_{41}^h(z_{4}-z_{1})  R_{42}^h(z_{4}-z_{2})  R_{43}^h(z_{4}-z_{3}) =0
\end{split}
\end{equation}
\paragraph{Example. $k=2$ $N=5$:}
$$
\begin{array}{c}
   R_{23}^-       R_{24}^-  R_{25}^-R_{13}^-   R_{14}^-       R_{15}^-  R_{51}  R_{41} R_{31}  R_{52}  R_{42}     R_{32}    - R_{23}     R_{24} R_{25} R_{13}   R_{14}      R_{15}  R_{51}^-     R_{41}^-  R_{31}^-     R_{52}^-  R_{42}^-     R_{32}^- \\
  +   R_{23}     R_{32}^-  R_{34}^- R_{35}^-  R_{12}^-      R_{14}^-   R_{15}^-     R_{51}  R_{41}  R_{21}     R_{53}  R_{43}     - R_{34} R_{35}    R_{12}    R_{14}     R_{15}  R_{51}^-     R_{41}^-  R_{21}^-     R_{53}^-  R_{43}^-     R_{23}^-  R_{32} \\
   +R_{34}     R_{24}  R_{42}^- R_{43}^-  R_{45}^-     R_{12}^-    R_{13}^-    R_{15}^-  R_{51}  R_{31}     R_{21}  R_{54}     -   R_{45}R_{12}     R_{13}  R_{15}     R_{51}^-  R_{31}^-     R_{21}^-  R_{54}^-     R_{34}^-  R_{24}^-     R_{42}  R_{43} \\
 +R_{45}     R_{35}  R_{25}  R_{52}^-    R_{53}^- R_{54}^-    R_{12}^-    R_{13}^-      R_{14}^-  R_{41}  R_{31}     R_{21}   - R_{12}     R_{13}  R_{14}  R_{41}^-     R_{31}^-  R_{21}^-     R_{45}^-  R_{35}^-     R_{25}^-  R_{52}  R_{53}     R_{54} \\
  + R_{12}     R_{13}  R_{31}^-     R_{21}^-  R_{34}^-  R_{35}^-     R_{24}^-     R_{25}^-  R_{52}  R_{42}     R_{53}  R_{43}     -R_{34}  R_{35}    R_{24}   R_{25}     R_{52}^-  R_{42}^-    R_{53}^-     R_{43}^-   R_{12}^-  R_{13}^-     R_{31}  R_{21} \\
  + R_{12}     R_{34}  R_{14}  R_{41}^-  R_{43}^-      R_{21}^-  R_{45}^-    R_{23}^-     R_{25}^-  R_{52}  R_{32}     R_{54}   - R_{45} R_{23}      R_{25}  R_{52}^-     R_{32}^-   R_{54}^-  R_{12}^-    R_{34}^-     R_{14}^-  R_{41}     R_{43}R_{21} \\
   + R_{12}     R_{45}  R_{35}  R_{15}     R_{51}^-     R_{53}^-   R_{54}^- R_{21}^-   R_{23}^-    R_{24}^-     R_{42}  R_{32}     - R_{23}     R_{24}  R_{42}^-     R_{32}^-  R_{12}^-     R_{45}^-  R_{35}^-     R_{15}^-  R_{51} R_{53}  R_{54} R_{21} \\
  + R_{23}     R_{13}  R_{24}  R_{14}     R_{41}^-   R_{42}^-   R_{31}^-   R_{32}^-     R_{45}^-  R_{35}^-     R_{53}  R_{54}   -  R_{45}     R_{35}  R_{53}^-   R_{54}^-     R_{23}^-  R_{13}^-    R_{24}^-     R_{14}^-  R_{41}R_{42}   R_{31}      R_{32} \\
  +     R_{23}     R_{13}  R_{45}  R_{25}     R_{15}  R_{51}^-    R_{52}^-  R_{54}^-    R_{31}^-    R_{32}^-     R_{34}^-  R_{43}   -  R_{34}     R_{43}^-  R_{23}^-     R_{13}^-  R_{45}^-     R_{25}^-  R_{15}^-     R_{51}  R_{52}R_{54} R_{31}      R_{32}   \\
  +R_{34}     R_{24}  R_{14}  R_{35}     R_{25}  R_{15}  R_{51}^-  R_{52}^-   R_{53}^-      R_{41}^-    R_{42}^-    R_{43}^-  - R_{34}^-     R_{24}^-  R_{14}^-     R_{35}^-  R_{25}^-     R_{15}^-  R_{51}     R_{52}  R_{53}   R_{41}  R_{42}    R_{43} =
  \end{array}
  $$
  \begin{equation}\label{rel25}
  =0
   \end{equation}
\section{Proof of $R$-matrix identities}\label{sect4}
\setcounter{equation}{0}

\subsection{Proof for $k=1$ case through AYBE}\label{sect41}
Our purpose is to prove the following identity for any $N\in\mZ_+$:
 \beq\label{a1}
  \begin{array}{c}
  \displaystyle{
 \sum\limits_{k=1}^N\overrightarrow{\prod\limits_{i=k+1}^N} R_{ki}^\hbar(z_k-z_i)
 \overleftarrow{\prod\limits_{j:j\neq k}^N} R_{jk}^\hbar(z_j-z_k-\eta)
 \overrightarrow{\prod\limits_{l=1}^{k-1}} R_{kl}^\hbar(z_k-z_l)
 -
 }
 \\ \ \\
  \displaystyle{
 -\sum\limits_{k=1}^N\overleftarrow{\prod\limits_{l=1}^{k-1}} R_{lk}^\hbar(z_l-z_k)
 \overrightarrow{\prod\limits_{j:j\neq k}^N} R_{kj}^\hbar(z_k-z_j-\eta)
 \overleftarrow{\prod\limits_{i=k+1}^N} R_{ik}^\hbar(z_i-z_k)=0
 \,,
 }
 \end{array}
 \eq
or equivalently,
 \beq\label{a2}
  \begin{array}{c}
  \displaystyle{
 \sum\limits_{k=1}^N R_{k,k+1}^\hbar\dots R_{k,N}^\hbar
 \cdot R_{N,k}^{\hbar,-}\dots R_{k+1,k}^{\hbar,-}R_{k-1,k}^{\hbar,-}\dots R_{1,k}^{\hbar,-}
 \cdot R^\hbar_{k,1}\dots R^\hbar_{k,k-1}
 -
 }
 \\ \ \\
  \displaystyle{
 -\sum\limits_{k=1}^N R_{k-1,k}^\hbar\dots R_{1,k}^\hbar
 \cdot R_{k,1}^{\hbar,-}\dots R_{k,k-1}^{\hbar,-}R_{k,k+1}^{\hbar,-}\dots R_{k,N}^{\hbar,-}
 \cdot R^\hbar_{N,k}\dots R^\hbar_{k+1,k}=0\,,
 }
 \end{array}
 \eq
where notations $R^\hbar_{ij}=R^\hbar_{ij}(z_i-z_j)$ and $R^{\hbar,-}_{ij}=R^\hbar_{ij}(z_i-z_j-\eta)$ are used for shortness. The multiplication points $\cdot$ between products are just for better visibility.

\subsubsection*{\underline{Addition formula of higher order.}}
For the proof of (\ref{a1}) we are going to use the following higher order addition formula (see (\ref{e21}) in the Appendix):
  \beq\label{a3}
  \begin{array}{l}
  \displaystyle{
R_{a,1}^{y_1}(x_1)R_{a,2}^{y_2}(x_2)\dots R_{a,N}^{y_N}(x_N)=
 }
 \\ \ \\
   \displaystyle{
 =R_{a,N}^{Y}(x_N)\cdot R^{y_1}_{N,1}(x_1-x_N)R^{y_2}_{N,2}(x_2-x_{N})\dots R^{y_{N-1}}_{N,N-1}(x_{N-1}-x_N)
 }
  \\ \ \\
   \displaystyle{
 +R_{N-1,N}^{y_N}(x_N-x_{N-1})\cdot R_{a,N-1}^Y\cdot(x_{N-1})R^{y_1}_{N-1,1}(x_1-x_{N-1})
 \dots R^{y_{N-2}}_{N-1,N-2}(x_{N-2}-x_{N-1})
 }
  \\ \ \\
   \displaystyle{
 +R_{N-2,N-1}^{y_{N-1}}(x_{N-1}-x_{N-2})R_{N-2,N}^{y_{N}}(x_{N}-x_{N-2})\cdot
 R_{a,N-2}^Y(x_{N-2})\cdot
  }
  \\ \ \\
   \displaystyle{
 \hspace{60mm}\cdot R^{y_1}_{N-2,1}(x_1-x_{N-2})
 \dots R^{y_{N-3}}_{N-2,N-3}(x_{N-3}-x_{N-2})
 }
   \\ \ \\
   \displaystyle{
 +\ldots+
 }
    \\ \ \\
   \displaystyle{
 +R_{1,2}^{y_2}(x_2-x_1)R_{1,3}^{y_3}(x_3-x_1)\dots R_{1,N}^{y_N}(x_{N}-x_{1})\cdot R_{a,1}^Y(x_1)
 \,,
 }
 \end{array}
 \eq
where $Y=\sum\limits_{m=1}^N
y_m$. Here $a$ is viewed as some index, which does not belong to the set $\{1,...,N\}$, so that
 the identity is in $\MatM^{\otimes(N+1)}$. Being written in the compact form (\ref{a3}) is as follows:
 \beq\label{a4}
  \begin{array}{c}
  \displaystyle{
 \overrightarrow{\prod\limits_{i=1}^N}
 R_{a,i}^{y_i}(x_i)=
 \sum\limits_{m=1}^N
 \overrightarrow{\prod\limits_{j=m+1}^{N}}
 R_{m,j}^{y_j}(x_j-x_{m})
 \cdot R_{a,m}^Y(x_m)\cdot
  \overrightarrow{\prod\limits_{j=1}^{m-1}}
 R_{m,j}^{y_j}(x_j-x_{m})\,.
 }
 \end{array}
 \eq

Consider also a special case, which is needed in our consideration. Let the index $a\in\{1,...,N\}$, so that
one of $R$-matrices in the product in the l.h.s. of (\ref{a3}) is skipped. Substitute also $y_i=\hbar$ and
$x_i=z_a-z_i-\eta$ for all $i\in\{1,...,a-1,a+1,...,N\}$. Then using notations of (\ref{a2}) we get
 \beq\label{a5}
  \begin{array}{c}
  \displaystyle{
  R^{\hbar,-}_{a,1}R^{\hbar,-}_{a,2}\dots R^{\hbar,-}_{a,a-1}R^{\hbar,-}_{a,a+1}\dots R^{\hbar,-}_{a,N-1}R^{\hbar,-}_{a,N}=
 }
 \\ \ \\
   \displaystyle{
  =\sum\limits_{m=1}^{a-1} R^\hbar_{m,m+1}\dots R^\hbar_{m,a-1}R^\hbar_{m,a+1}\dots R^\hbar_{m,N}
  \cdot R^{Y,-}_{a,m}\cdot R^\hbar_{m,1}\dots R^\hbar_{m,m-1}+
 }
  \\ \ \\
   \displaystyle{
  +\sum\limits_{m=a+1}^{N} R^\hbar_{m,m+1}\dots R^\hbar_{m,N}
  \cdot R^{Y,-}_{a,m} \cdot R^\hbar_{m,1}\dots R^\hbar_{m,a-1}R^\hbar_{m,a+1}\dots R^\hbar_{m,m-1}\,.
 }
 \end{array}
 \eq
Also, by considering inverse order in the set $i\in\{1,...,a-1,a+1,...,N\}$ we obtain
 \beq\label{a6}
  \begin{array}{c}
  \displaystyle{
  R^{\hbar,-}_{N,a}R^{\hbar,-}_{N-1,a}\dots R^{\hbar,-}_{a+1,a}R^{\hbar,-}_{a-1,a}\dots R^{\hbar,-}_{2,a}R^{\hbar,-}_{1,a}=
 }
 \\ \ \\
   \displaystyle{
  =\sum\limits_{m=1}^{a-1} R^\hbar_{m-1,m}\dots R^\hbar_{1,m}
  \cdot R^{Y,-}_{m,a} \cdot R^\hbar_{N,m}\dots R^\hbar_{a+1,m}R^\hbar_{a-1,m}\dots R^\hbar_{m+1,m}+
 }
  \\ \ \\
   \displaystyle{
  +\sum\limits_{m=a+1}^{N} R^\hbar_{m-1,m}\dots R^\hbar_{a+1,m}R^\hbar_{a-1,m}\dots R^\hbar_{1,m}
  \cdot R^{Y,-}_{m,a} \cdot R^\hbar_{N,m}\dots R^\hbar_{m+1,m}\,.
 }
 \end{array}
 \eq

\subsubsection*{\underline{Proof of (\ref{a1}):}}
Consider the middle products in both sides of (\ref{a1}), i.e. the products with $R$-matrices depending on $z_a-z_b-\eta$. Substitute (\ref{a5}) and (\ref{a6}) into the lower and upper sums respectively. It is easy to see that each sum is then transformed into the double sum containing $N(N-1)$ terms with all possible distinct indices for $R^{Y,-}_{ab}=R_{ab}^Y(z_a-z_b-\eta)$.

Our strategy is to fix $a,b\in\{1,...,N\}$ and compare two terms containing $R^{Y,-}_{ab}$ in both sums. Let $a<b$.
Then using the lower line in the r.h.s. of (\ref{a5}) and the upper line in the r.h.s of (\ref{a6}) we obtain the following to terms:
 \beq\label{a7}
  \begin{array}{c}
  \displaystyle{
  R^\hbar_{b,b+1}\dots R^\hbar_{b,N}\cdot R^\hbar_{a-1,a}\dots R^\hbar_{1,a}\cdot R_{ab}^{Y,-}\cdot
  R^\hbar_{N,a}\dots R^\hbar_{b+1,a}R^\hbar_{b-1,a}\dots R^\hbar_{a+1,a}\cdot R^\hbar_{b,1}\dots R^\hbar_{b,b-1}
 }
 \end{array}
 \eq
and
 \beq\label{a8}
  \begin{array}{c}
  \displaystyle{
  -R^\hbar_{a-1,a}\dots R^\hbar_{1,a}\cdot R^\hbar_{b,b+1}\dots R^\hbar_{b,N}\cdot  R_{ab}^{Y,-}\cdot
  R^\hbar_{b,1}\dots R^\hbar_{b,a-1}R^\hbar_{b,a+1}\dots R^\hbar_{b,b-1}\cdot
  R^\hbar_{N,a}\dots R^\hbar_{a+1,a}\,.
 }
 \end{array}
 \eq
Expressions to the left of $R_{ab}^{Y,-}$ are equal to each other since $R^\hbar_{b,b+1}\dots R^\hbar_{b,N}$
and $R^\hbar_{a-1,a}\dots R^\hbar_{1,a}$ commute due to $a<b$. Let us compare expressions to the right of $R_{ab}^{Y,-}$ and prove their coincidence. They are of the form (we write the last products more explicitly):
 \beq\label{a9}
  \begin{array}{c}
  \displaystyle{
  R^\hbar_{N,a}\dots R^\hbar_{b+1,a}\cdot R^\hbar_{b-1,a}\dots R^\hbar_{a+1,a}\cdot
  \underbrace{R^\hbar_{b,1}\dots R^\hbar_{b,a-1}}\cdot R^\hbar_{b,a}\cdot R^\hbar_{b,a+1}\dots R^\hbar_{b,b-1}
 }
 \end{array}
 \eq
and
 \beq\label{a10}
  \begin{array}{c}
  \displaystyle{
  R^\hbar_{b,1}\dots R^\hbar_{b,a-1}\cdot R^\hbar_{b,a+1}\dots R^\hbar_{b,b-1}\cdot
  \underbrace{R^\hbar_{N,a}\dots R^\hbar_{b+1,a}}\cdot R^\hbar_{b,a}\cdot R^\hbar_{b-1,a}\dots R^\hbar_{a+1,a}\,.
 }
 \end{array}
 \eq
The underbraced factors can be moved to the left. This provides
$R^\hbar_{N,a}\dots R^\hbar_{b+1,a}\cdot R^\hbar_{b,1}\dots R^\hbar_{b,a-1}$ as a common factor. Finally, we are left with the following two expressions:
 \beq\label{a11}
  \begin{array}{c}
  \displaystyle{
  R^\hbar_{b-1,a}\dots R^\hbar_{a+1,a}
  \cdot R^\hbar_{b,a}\cdot R^\hbar_{b,a+1}\dots R^\hbar_{b,b-1}
 }
 \end{array}
 \eq
and
 \beq\label{a12}
  \begin{array}{c}
  \displaystyle{
   R^\hbar_{b,a+1}\dots R^\hbar_{b,b-1}\cdot
   R^\hbar_{b,a}\cdot R^\hbar_{b-1,a}\dots R^\hbar_{a+1,a}\,.
 }
 \end{array}
 \eq
The products (\ref{a11}) and (\ref{a12}) are equal due to Yang-Baxter equation. Indeed,
consider (\ref{a11}). Due to the Yang-Baxter equation $R^\hbar_{a+1,a}
  \cdot R^\hbar_{b,a}\cdot R^\hbar_{b,a+1}=R^\hbar_{b,a+1}
  \cdot R^\hbar_{b,a}\cdot R^\hbar_{a+1,a}$. After applying the Yang-Baxter equation $R^\hbar_{b,a+1}$
  is moved to the left, while $R^\hbar_{a+1,a}$ is moved to the right. By doing so step by step one
  transform (\ref{a11}) into (\ref{a12}). Therefore, (\ref{a11}) and (\ref{a12}) coincide, and
  (\ref{a9}) and (\ref{a10}) coincide as well, so that the sum of (\ref{a7}) and (\ref{a8}) equals zero.
  This finishes the proof for $a<b$.

For $a>b$ the proof is similar.
In this case one should use
the upper line in the r.h.s. of (\ref{a5}) and the lower line in the r.h.s of (\ref{a6}). \hfill$\blacksquare$

\subsection{Proof for $k>1$}\label{sect42}

\paragraph{The strategy of the proof} is as follows. Consider the l.h.s.
of (\ref{a20})\footnote{In what follows we omit for brevity writing dependence on all set of variables.}:
 \beq\label{e01}
  \begin{array}{c}
    \displaystyle{
  {\mathcal F}={\mathcal F}^--{\mathcal F}^+\,,\quad {\mathcal F}^\pm={\mathcal F}^\pm(z_1,...,z_N,\eta,\hbar,\tau,k,N,M)\,,
 }
 \end{array}
 \eq
 \beq\label{e02}
  \begin{array}{c}
  \displaystyle{
  {\mathcal F}^+= \sum\limits_{1\leq i_1<...i_k\leq N}
  {\mathcal F}^+_{i_1,...,i_k}(k,N)\,,\quad
   {\mathcal F}^-= \sum\limits_{1\leq i_1<...i_k\leq N}
    {\mathcal F}^-_{i_1,...,i_k}(k,N)\,.
 }
 \end{array}
 \eq
It follows from (\ref{r05}) and (\ref{r051}) that the function $\mathcal F$
may have only simple poles in variables $z_1,...,z_N$, $\eta$ and may have higher order poles in $\hbar$.

The main (and the most technical part of the proof) is to show that the function $\mathcal F$ has no poles
at $z_a-z_b-\eta=0$ for any $a,b$ ($1\leq a\neq b\leq N$), that is $\mathcal F$ has no poles in variable $\eta$. Suppose it is done.
 We are going to show that $\mathcal F$ is independent of $\eta$. For this purpose consider the quasi-periodic behaviour of $\mathcal F$ with respect to variable $\eta$ on the ''large'' torus with periods $M$ and $M\tau$ (instead of $1$ and $\tau$ on the original one).
 Due to $Q^M=\Lambda^M=1_M$ (\ref{a041}) and the properties (\ref{r721})  we conclude that for any $i,j$ ($1\leq i\neq j\leq N$)
 \beq\label{r725}
 \begin{array}{c}
  \displaystyle{
 R_{ij}^\hbar(x+M)=R_{ij}^\hbar(z)\,,
\qquad
 R_{ij}^\hbar(x+M\tau)=\exp(-2\pi\imath\hbar)\,R_{ij}^\hbar(x)\,.
  }
 \end{array}
 \eq
 Each term ${\mathcal F}^\pm_{i_1,...,i_k}(k,N)$ contains a product of $k(N-k)$ $R$-matrices $R^-$ (the middle lines in the r.h.s. of (\ref{a21}) and (\ref{a22})). Therefore,
 \beq\label{r726}
 \begin{array}{c}
  \displaystyle{
 {\mathcal F}(\eta+M)={\mathcal F}(\eta)\,,\qquad {\mathcal F}(\eta+M\tau)=
 \exp(-2\pi\imath k(N-k)\hbar)\,{\mathcal F}(\eta)\,.
  }
 \end{array}
 \eq
 Notice that by enlarging the torus we get additional poles in $\eta$. For any $a,b$ ($1\leq a\neq b\leq N$) the function $\mathcal F$ may have simple poles at $M^2$ points $\eta=z_b-z_a+\Omega_{m_1,m_2}$ on the large torus, where
 \beq\label{r7261}
 \begin{array}{c}
  \displaystyle{
 \Omega_{m_1,m_2}=m_1+m_2\tau\,,\quad 0\leq m_1,m_2\leq M-1\,.
  }
 \end{array}
 \eq
 Thus we need to extend the proof of the absence of poles at $\eta=z_b-z_a$ to all points $\eta=z_b-z_a+\Omega_{m_1,m_2}$. Suppose it is also done.

 Now we are in position to use the following basic fact on elliptic functions.  Consider an entire
 function $f$
 on elliptic curve with periods $1$ and $\tau_0$. Suppose it has the quasi-periodic behaviour given by
 \beq\label{r727}
 \begin{array}{c}
  \displaystyle{
 f(x+\tau_0)=\exp(-2\pi\imath(A_1+A_2 x))f(x)\,,\qquad f(x+1)=\exp(-2\pi\imath A_3)f(x)\,.
  }
 \end{array}
 \eq
 Then $A_2$ is necessarily a positive integer, and $f$ has $A_2$ zeros in the fundamental parallelogram of the elliptic curve.

 By comparing (\ref{r726}) and (\ref{r727}) we conclude that $A_2=0$. Thus, ${\mathcal F}$ is a constant as a function of $\eta$.
When $\mathcal F$ is shown to be independent of $\eta$, we  may put $\eta=0$. If $\eta=0$ the statement of the theorem becomes simple. Indeed, due to the unitarity property (\ref{q03}) all the products of $R$-matrices (in any term of $\mathcal F$) are transformed into the identity matrix ${\rm Id}=1_{M^N}$ multiplied by the corresponding products of scalar functions $\phi$. That is
 \beq\label{r728}
 \begin{array}{c}
  \displaystyle{
 {\mathcal F}|_{\eta=0}=\Big({\rm l.h.s.\ of\ (\ref{IdenRuij}) }\Big)|_{\eta=0}\times {\rm Id}=0\,.
  }
 \end{array}
 \eq
 In this way Theorem \ref{th2} will be
 proved.

 To summarize, we need to prove the following two statements. The first one is
 \begin{predl}\label{prop1}
 For any $a,b$ ($1\leq a\neq b\leq N$) the following set of identities holds:
 \beq\label{r750}
 \begin{array}{c}
  \displaystyle{
  \res\limits_{z_a=z_b+\eta}{\mathcal F}=0\,.
  }
 \end{array}
 \eq
 \end{predl}
 The second one is a generalization of (\ref{r750}) to all points (\ref{r7261}).
 \begin{predl}\label{prop2}
 For any $a,b$ ($1\leq a\neq b\leq N$) and any $\Omega_{m_1,m_2}$ (\ref{r7261}) the following set of identities holds:
 \beq\label{r751}
 \begin{array}{c}
  \displaystyle{
  \res\limits_{\z_a=z_b+\eta-\Omega_{m_1,m_2}}{\mathcal F}=0\,.
  }
 \end{array}
 \eq
 \end{predl}
The proof of Proposition \ref{prop2} is based on  the  Proposition \ref{prop1} and the quasi-periodic behaviour of $R$-matrices. It is given the Appendix.

 The proof of Proposition \ref{prop1} is by induction in $k$. For $k=1$ the identity was proved in
 the previous subsection. Therefore, (\ref{r750}) is valid for $k=1$. Consider $k>1$ and let (\ref{r750})
  holds true for $k-1$. Then we need to prove (\ref{r750}) for $k$.
 Let us fix indices $a$ and $b$. For definiteness, assume $a<b$ (the case $a>b$ is considered similarly).
 Consider the function ${\mathcal F}^+(k,N)$. The set of terms in ${\mathcal F}^+(k,N)$, which contain
 the pole at $z_a=z_b+\eta$, consists of those ${\mathcal F}_{i_1,...,i_k}^+(k,N)$, where $i_d=b$ for some $d$,
  $1\leq d\leq k$. Similarly, the set of terms in ${\mathcal F}^-(k,N)$, which contain
 the pole at $z_a=z_b+\eta$, consists of those ${\mathcal F}_{i_1,...,i_k}^-(k,N)$, where $i_d=a$ for some $d$,
  $1\leq d\leq k$.
 The proof of
 (\ref{r750}) follows from the following statement.
 \begin{lemma}\label{lprop1}
For any $d$ ($1\leq d\leq k$) and any $i_1,...,i_k$ ($1\leq i_1<...<i_k\leq N$)

 1) Let $i_d=b$. Then
 \beq\label{r752}
 \begin{array}{c}
  \displaystyle{
  \res\limits_{z_a=z_b+\eta}{\mathcal F}_{i_1,...,i_k}^+(\{z_1,...,z_N\}, k,N)=
  }
  \\ \ \\
   \displaystyle{
  ={\mathcal A}(a,b)\cdot {\mathcal F}_{i_1,...,i_{d-1},i_{d+1},...,i_k}^+(\{z_1,...,z_N\}\setminus\{z_a,z_b\}, k-1,N-2)
  \cdot P_{ab}
  \cdot {\mathcal B}(a,b)\,\Big|_{z_a=z_b+\eta}\,,
  }
 \end{array}
 \eq
 where
 \beq\label{r753}
 \begin{array}{c}
  \displaystyle{
  {\mathcal A}(a,b)=R_{b,b+1}...R_{b,N}\cdot R_{a-1,a}...R_{1,a}\,,
  }
 \end{array}
 \eq
 \beq\label{r754}
 \begin{array}{c}
  \displaystyle{
  {\mathcal B}(a,b)=R_{b,1}\dots R_{b,a-1}\cdot R_{b,a+1}\dots R_{b-1,b}\cdot R_{N,a}\dots R_{a+1,a}
  }
 \end{array}
 \eq
 and ${\mathcal F}_{i_1,...,i_{d-1},i_{d+1},...,i_k}^+(\{z_1,...,z_N\}\setminus\{z_a,z_b\}, k-1,N-2)$ is the
 term in ${\mathcal F}^+$ written for the set of $N-2$ variables $z_1,...,z_{a-1},z_{a+1},...,z_{b-1},z_{b+1},...,z_N$
 and for $k-1$.

 2) Let $i_d=a$. Then
 \beq\label{r755}
 \begin{array}{c}
  \displaystyle{
  \res\limits_{z_a=z_b+\eta}{\mathcal F}_{i_1,...,i_k}^-(\{z_1,...,z_N\}, k,N)=
  }
  \\ \ \\
   \displaystyle{
  ={\mathcal A}(a,b)\cdot {\mathcal F}_{i_1,...,i_{d-1},i_{d+1},...,i_k}^-(\{z_1,...,z_N\}\setminus\{z_a,z_b\}, k-1,N-2)
  \cdot P_{ab}
  \cdot {\mathcal B}(a,b)\,,
  }
 \end{array}
 \eq
 where
 ${\mathcal F}_{i_1,...,i_{d-1},i_{d+1},...,i_k}^-(\{z_1,...,z_N\}\setminus\{z_a,z_b\}, k-1,N-2)$ is the
 term in ${\mathcal F}^-$ written for the set of $N-2$ variables $z_1,...,z_{a-1},z_{a+1},...,z_{b-1},z_{b+1},...,z_N$
 and for $k-1$. The expressions ${\mathcal A}(a,b)$ and ${\mathcal B}(a,b)$ are given by (\ref{r753}) and (\ref{r754}) respectively.
 \end{lemma}
 The proof of this lemma is technical. It is based on applying the quantum Yang-Baxter equation. We give the proof in the Appendix.

  Notice that the expressions ${\mathcal A}(a,b)$ and ${\mathcal B}(a,b)$ (\ref{r753}) and (\ref{r754}) are the same in (\ref{r752}) and (\ref{r755}). Another important
  remark is that ${\mathcal A}(a,b)$ and ${\mathcal B}(a,b)$ are independent of the choice of $i_1,...,i_k$. Therefore, we conclude from the above lemma that
 \beq\label{r756}
 \begin{array}{c}
  \displaystyle{
  \res\limits_{z_a=z_b+\eta}{\mathcal F}(k,N)
  ={\mathcal A}(a,b)\cdot {\mathcal F}(k-1,N-2)
  \cdot {\mathcal B}(a,b)\,.
  }
 \end{array}
 \eq
 The expression in the r.h.s. vanishes since ${\mathcal F}(k-1,N-2)=0$ by the induction assumption. This finishes the proof of the Proposition \ref{prop1} and Theorem \ref{th2}. \hfill $\blacksquare$

\section{Applications and discussion}\label{sect5}
\setcounter{equation}{0}

\subsubsection*{Limit to differential operators.} The elliptic Macdonald-Ruijsenaars operators (\ref{Dscalar})
can be viewed as deformations of Hamiltonians for the quantum elliptic Calogero-Moser model \cite{OP}.
Namely, make the substitutions $\hbar\rightarrow \epsilon\hbar$, $\eta\rightarrow \epsilon\eta$ in the first Macdonald-Ruijsenaars operator $D_1(\hbar,\eta)\rightarrow D_1^\epsilon=D_1(\epsilon\hbar,\epsilon\eta)$
 and consider the limit
 $\epsilon\rightarrow 0$:
 \beq\label{y01}
 \begin{array}{c}
  \displaystyle{
  D_1^\epsilon=\Big(\frac{\vth(\epsilon\hbar)}{\vth'(0)}\Big)^{N-1}
  \sum\limits_{k=1}^N\prod\limits_{i:i\neq k}^N\frac{\vth(z_i-z_k+\epsilon\hbar)}{\vth(z_i-z_k)}\,e^{-\epsilon\eta\p_{z_k}}=
  N-\epsilon\sum\limits_{k=1}^N\eta\p_{z_k}+\epsilon^2 H_2^{\rm CM}+O(\epsilon^3)\,,
  }
 \end{array}
 \eq
 \beq\label{y02}
 \begin{array}{c}
  \displaystyle{
  H_2^{\rm CM}=\frac12\sum\limits_{k=1}^N  v_k^2-\frac{\hbar(\hbar-\eta)}{2}\sum\limits_{i\neq j}^N \wp(z_i-z_j)\,,
  \qquad v_k=\eta\p_{z_k}-\hbar\sum\limits_{i:i\neq k}^N E_1(z_i-z_k)\,,
  }
 \end{array}
 \eq
 where we used the first Eisenstein function $E_1$ and the Weierstrass $\wp$-function defined as in (\ref{serE})-(\ref{serE2}).
 The same limit for the spin operator ${\mathcal D}_1$ yields:
 \beq\label{y03}
 \begin{array}{c}
  \displaystyle{
  {\mathcal D}_1^\epsilon=
  N\,{\rm Id}-\epsilon\,{\rm Id}\sum\limits_{k=1}^N\eta\p_{z_k}+\epsilon^2 H^{\rm tops}_2+O(\epsilon^3)\,,\qquad
  {\rm Id}=1_{M^N}\,,
  }
 \end{array}
 \eq
 \beq\label{y04}
 \begin{array}{c}
  \displaystyle{
  H_2^{\rm tops}={\rm Id}\Bigg(\frac12\sum\limits_{k=1}^N  v_k^2
  +\hbar\eta\frac{N(N-1)}{6}\frac{\vth'''(0)}{\vth'(0)}
  -\frac{\hbar^2}{2}\sum\limits_{i\neq j}^N \wp(z_i-z_j)
  \Bigg)-\frac{\hbar\eta}{2}\sum\limits_{i\neq j}^N \p_{z_i}r_{ij}(z_i-z_j)\,,
  }
 \end{array}
 \eq
 where $r_{ij}(z_i-z_j)$ is the classical elliptic (Belavin-Drinfeld \cite{BD}) $r$-matrix (\ref{r053}) satisfying
 the classical Yang-Baxter equation (\ref{r054}). In the scalar case ($M=1$) $r$-matrix $r_{12}(z)$ becomes the function $E_1(z)$ and (\ref{y04}) turns into (\ref{y02}). Let us mention that the pairwise interaction
  given by $\p_{z_i}r_{ij}(z_i-z_j)$ differs from the one $P_{ij}\p_{z_i}E_1(z_i-z_j)$  known for the spin Calogero-Moser model. This is why (\ref{y04}) is an anisotropic version of the spin Calogero-Moser model, and this is
  why we call ${\mathcal D}_k$ anisotropic spin Macdonald-Ruijsenaars operators. The operator (\ref{y04})
  was obtained in \cite{GZ,GSZ} as quantization (in the fundamental representation of ${\rm gl}_M$ Lie algebra) of the Hamiltonian for the model of $N$ interacting ${\rm gl}_M$ elliptic tops. Using our results it becomes a straightforward calculation to get the higher commuting Hamiltonians. In this respect the obtained
  set of commuting operators ${\mathcal D}_k$ describes the difference (relativistic) generalization for the model of interacting tops.

\subsubsection*{Trigonometric and rational limits.}
Trigonometric limits of the elliptic $R$-matrix provide a wide class of $R$-matrices \cite{AHZ}. Besides straightforward application of limit ${\rm Im}(\tau)\rightarrow \infty$ to the elliptic $R$-matrix (\ref{BB}) one could also
perform a gauge transformation (depending on the moduli $\tau$) before the limit. Then the result of the limit provides a variety of different answers. The gauge transformation ${\bar R}_{12}^\hbar(z)\rightarrow G^{(1)}G^{(2)}{\bar R}_{12}^\hbar(z)(G^{(1)})^{-1}(G^{(2)})^{-1}={\widetilde R}_{12}^\hbar(z)$ is defined by the matrix $G\in\MatM$, which is independent of the spectral parameter $z$\footnote{We discuss possible gauge transformations in our next paper \cite{MZ2}.}. Then, one can consider (here we use the notation $G^{(i)}$ as in (\ref{a903})) the transformed set of operators
 \beq\label{y05}
 \begin{array}{c}
  \displaystyle{
  \widetilde{\mathcal D}_k=G^{(1)}\ldots G^{(N)}{\mathcal D}_k
  (G^{(1)})^{-1}\ldots(G^{(N)})^{-1}\,.
  }
 \end{array}
 \eq
 It is easy to see that these operators are obtained from
 ${\mathcal D}_k$ by replacing ${\bar R}_{ij}$ with the gauge transformed $R$-matrices ${\widetilde R}_{ij}$. Also, $\widetilde{\mathcal D}_k$ obviously commute with each other. If all ${\widetilde R}_{ij}$ have finite trigonometric limit, then $\widetilde{\mathcal D}_k$ have
 finite limit as well and the limiting operators commute. Thus, all trigonometric $R$-matrices from \cite{AHZ} can be used in our construction and the corresponding spin operators are commuting. The functions $\phi$ in the definition of operators (\ref{q20}) should be used in their trigonometric form (\ref{a081}). Similar reasoning allows to include into consideration a class of rational $R$-matrices, obtained analogously from the elliptic $R$-matrix as it was performed in \cite{Sm,LOZ14}.

 In \cite{KZ19} it was shown that the trigonometric $R$-matrices obtained in \cite{AHZ} could be included into classification of trigonometric solutions of the associative Yang-Baxter equation (\ref{AYBE}) with some additional properties \cite{Pol2,T}. In this way we come to a natural conjecture: {\em the statements of Theorems \ref{th1} and \ref{th2} are valid for any $R$-matrix (we assume $R$-matrices depending on the difference of spectral parameters) satisfying the quantum Yang-Baxter equation (\ref{QYB}), the associative Yang-Baxter equation (\ref{AYBE}) and the unitarity property (\ref{q03})-(\ref{q05}).} In this way one can include into consideration a wide class of trigonometric and rational solutions of (\ref{AYBE}). In fact, we have already proved a part of this conjecture since our proof of $R$-matrix identities for $k=1$ was based on higher order analogue of (\ref{AYBE}) and did not use explicit form of $R$-matrix. So one needs to perform a similar proof for $k>1$.

 The above mentioned trigonometric and rational $R$-matrices in ${\rm GL}_2$ case were obtained in \cite{Chered2}. In the
 trigonometric case they include the XXZ ${\rm GL}_2$ $R$-matrix and its 7-vertex deformation. In the rational case they include the standard Yang's rational $R$-matrix
 $R_{12}^\hbar(z)=\hbar^{-1}{\rm Id}+z^{-1}P_{12}$ and its 11-vertex deformation. For $M>2$ there is a problem with one important case -- ${\rm GL}_M$ XXZ $R$-matrix for the affine quantized algebra ${\hat{\mathcal U}}_q({\rm gl}_M)$ does not satisfy (for $M>2$) the associative Yang-Baxter equation although  the disparity is independent of spectral parameters (see \cite{KZ19}), and finally such $R$-matrix satisfy the $R$-matrix identities (\ref{a20})-(\ref{a22}). We prove it accurately
 in our next paper \cite{MZ2}.
 That is, the above conjecture does not cover all possible
 trigonometric $R$-matrices satisfying (\ref{a20})-(\ref{a22}).
 We will clarify these questions in our future works.

\subsubsection*{Classical limit.} %
The spin generalization of the elliptic Ruijsenaars-Schneider model was introduced  in \cite{KrichZ} at the level of Lax equations in classical mechanics. It is an open problem to describe its Poisson and $r$-matrix structures and their quantizations. At the same time much progress was achieved in studies of the trigonometric models \cite{AO,Feher,ChF}.
In this paper we proposed anisotropic version for elliptic spin Macdonald-Ruijsenaars operators.
Since the differential operator (\ref{y04}) is the Hamiltonian for the model of interacting tops it is natural to guess that the  classical model
corresponding to difference operators ${\mathcal D}_k$ is the integrable model of relativistic interacting tops \cite{Z19}. This model is described
by the Lax equations related to ${\rm SL}_{NM}$-bundles over elliptic curve with non-trivial characteristic classes \cite{LOSZ}. The quantization of its phase space is given by a mixture of elliptic quantum group and Sklyanin algebra \cite{SeZ2}. At the same time the problem of constructing commuting Hamiltonians looks highly nontrivial in that approach. In present paper we avoid these difficulties by straightforward definition of operators. This allowed us to reduce the problem to derivation and
proof of the identities.

Let us also mention that in the rational case with the Yang's  rational $R$-matrix
 $R_{12}^\hbar(z)=\hbar^{-1}{\rm Id}+z^{-1}P_{12}$
 the constructed operators provide quantization of the custom (not anisotropic) quantum spin Ruijsenaars-Schneider model. Therefore, its classical limit is the rational spin Ruijsenaars-Schneider model introduced in \cite{KrichZ}.

\subsubsection*{Long range spin chains.}
There is a class of integrable spin chains obtained from
spin generalizations of the Calogero-Moser-Sutherland models. The Hamiltonians are obtained by the so-called ''Polychronakos freezing trick'', which works as a recipe: remove the differential operators and fix the variables $z_i$ as equidistant positions on a circle. In our notations the latter means
to fix $z_i=x_i$ ($i=1,...,N$) with
 \beq\label{y06}
 \begin{array}{c}
  \displaystyle{
  x_1=\frac{1}{N}\,,\ x_2=\frac{2}{N}\,,\ldots \,, x_N=1\,.
  }
 \end{array}
 \eq
The commutativity of Hamiltonians obtained in this way is
a highly non-trivial problem. For example, the Haldane-Shastry Hamiltonians (those obtained by freezing from the spin Calogero-Moser-Sutherland model) commute due to
underlying Yangian symmetry
\cite{BGHP}. Anisotropic versions for these spin chains are known as well \cite{Lam,SeZ}.
The application of freezing trick to difference operators
provides q-deformations of the above mentioned models. This way first suggested in \cite{Uglov} for ${\rm GL}_2$ XXZ $R$-matrix. This model is studied in detail in recent work \cite{LPS}.

The results of Theorems \ref{th1} and \ref{th2} allow to suggest a more general model based on elliptic ${\rm GL}_M$ $R$-matrix for any $M$.
 We present the expression for the first Hamiltonian:
 \beq
 H_1=\left(\sum_{i=1}^N \prod\limits_{\substack{j=1\\j\neq i}}^N\phi(z_j-z_i)\sum_{k=1}^{i-1}\bar{R}_{i-1,i}\dots \bar{R}_{k+1,i}\bar{R}_{k,i}\left(\frac{\partial}{\partial z_i}\bar{R}_{i,k}\right)\bar{R}_{i,k+1}\dots \bar{R}_{i,i-1}\right)\Bigg|_{z_1=x_1,...,z_N=x_N}.
 \eq
Higher commuting Hamiltonians can be constructed in a similar way.
This model can be considered as anisotropic q-deformation of the elliptic Inozemtsev chain \cite{Inoz}. More precisely,
it is a q-deformation of the anisotropic chain introduced in \cite{SeZ}.
 All details can be found in \cite{MZ2}.

\section{Appendices}\label{sect6}
\subsection{Appendix A: elliptic functions and elliptic $R$-matrix}\label{sec:A}
\def\theequation{A.\arabic{equation}}
\setcounter{equation}{0}

\paragraph{Elliptic functions.} Using the theta-function
\beq\label{a0963}\begin{array}{c} \displaystyle{
     \vartheta (z)=\vartheta (z|\tau) = -\sum_{k\in \mathbb{Z}} \exp \left( \pi \imath \tau (k + \frac{1}{2})^2 + 2\pi \imath (z + \frac{1}{2}) (k + \frac{1}{2}) \right)\,,\quad {\rm Im}(\tau)>0\,.
}\end{array}\eq
It is odd $\vth(-z)=-\vth(z)$ and has simple zero at $z=0$.
Next, define
the  Kronecker elliptic function \cite{Weil}:
\beq\displaystyle{
\label{a0962}
    \phi(z, u) =
            \frac{\vartheta'(0) \vartheta (z + u)}{\vartheta (z) \vartheta (u)}=\phi(u,z)\,.
}\eq
 As function of $z$ it has simple pole at $z=0$ and
 \beq\label{a095}
  \begin{array}{l}
  \displaystyle{
\res\limits_{z=0}\phi(z,u)=1\,.
 }
 \end{array}
 \eq
 The local expansions near $z=0$ takes the form:
\beq\label{serphi}
\begin{array}{c} \displaystyle{
    \phi(z, u) = \frac{1}{z} + E_1 (u) + z\,\frac{E^2_1(u) - \wp(u)}{2} + O(z^2)\,,
}\end{array}\eq
where
\beq\label{serE}
\begin{array}{c}
\displaystyle{
    E_1(z) = \partial_z \ln \vartheta(z) =\zeta(z)+ \frac{z}{3} \frac{\vartheta'''(0) }{\vartheta'(0)}=\frac{1}{z} + \frac{z}{3} \frac{\vartheta'''(0) }{\vartheta'(0)} + O(z^3)\,.
}\end{array}\eq
\beq\label{serE2}
\begin{array}{c}
\displaystyle{
    E_2(z)=-E_1'(z) = -\partial^2_z \ln \vartheta(z) =
    \wp(z) - \frac{1}{3}\frac{\vartheta'''(0) }{\vartheta'(0)}=
    \frac{1}{z^2}- \frac{1}{3}\frac{\vartheta'''(0) }{\vartheta'(0)}
    +O(z^2)\,.
}\end{array}\eq
Here $\wp(z)$ and $\zeta(z)$ are the Weierstrass $\wp$- and $\zeta$-functions respectively.

The behaviour on the lattice of periods $\Gamma=\mZ\oplus\mZ\tau$ (of elliptic curve $\mC/\Gamma$) is as follows:
 \beq\label{a0961}
  \begin{array}{l}
  \displaystyle{
 \vth(z+1)=-\vth(z)\,,\qquad \vth(z+\tau)=-e^{-\pi\imath\tau-2\pi\imath z}\vth(z)\,,
 }
 \end{array}
 \eq
 that is
 \beq\label{a096}
  \begin{array}{l}
  \displaystyle{
 \phi(z+1,u)= \phi(z,u)\,,\qquad  \phi(z+\tau,u)= e^{-2\pi\imath u}\phi(z,u)\,.
 }
 \end{array}
 \eq
 The function (\ref{a0962}) satisfies the following addition formula (genus one Fay identity):
%
\beq \begin{array}{c} \label{Fay} \displaystyle{
    \phi(z_1, u_1) \phi(z_2, u_2) = \phi(z_1, u_1 + u_2) \phi(z_2 - z_1, u_2) + \phi(z_2, u_1 + u_2) \phi(z_1 - z_2, u_1)
}\end{array}\eq
and the identity
\beq\begin{array}{c} \displaystyle{
\label{a0964}
    \phi(z, u) \phi(z, -u) = \wp (z) - \wp (u)\,.
}\end{array}\eq
%
Being written in terms of theta function (\ref{a0963}) the identity (\ref{Fay})
turns into one of Riemann identities for theta functions, see e.g. \cite{Mum}.

In the trigonometric limit ${\rm Im}(\tau)\rightarrow +\infty$ we have
$\vth(z)=2\exp(\frac{\pi\imath\tau}{4})\sin(\pi z)+O(\exp(\frac{9\pi\imath\tau}{4}))$. Therefore, the trigonometric and rational limits of the function (\ref{a0962}) are as follows:
 \beq\label{a081}
 \begin{array}{c}
  \displaystyle{
\phi^{\rm trig}(z,u)=\pi\cot(\pi z)+\pi\cot(\pi u)\,,\qquad
\phi^{\rm rat}(z,u)=1/z+1/u\,.
 }
 \end{array}
 \eq

Define also the following set of $M^2$ functions:
 \beq\label{a08}
 \begin{array}{c}
  \displaystyle{
 \vf_a(z,\om_a+\hbar)=\exp(2\pi\imath\frac{a_2z}{M})\,\phi(z,\om_a+\hbar)\,,\quad
 \om_a=\frac{a_1+a_2\tau}{M}\,,
 }
 \end{array}
 \eq
where $M\in\mZ_+$ is an integer number and $a=(a_1, a_2)\in\mZ_M\times\mZ_M$.

\paragraph{Matrix basis.}
For construction of elliptic $R$-matrix the following matrix basis in $\MatM$ is used:
\beq\label{a971}\begin{array}{c} \displaystyle{
        T_\al = \exp \left( \al_1 \al_2 \frac{\pi \imath}{M} \right) Q^{\al_1} \Lambda^{\al_2}, \quad \al = (\al_1, \al_2)\in \mZ_M \times \mZ_M\,,
}\end{array}\eq
where $Q,\Lambda\in\MatM$ with entries
\beq\label{a041}
 \begin{array}{c}
  \displaystyle{
 Q_{kl}=\delta_{kl}\exp(\frac{2\pi
 \imath}{{ M}}k)\,,\ \ \
 \Lambda_{kl}=\delta_{k-l+1=0\,{\hbox{\tiny{mod}}}\,
 { M}}\,,\quad Q^{ M}=\Lambda^{ M}=1_M
 }
 \end{array}
 \eq
 are the finite-dimensional representations of the Heisenberg group:
 \beq\label{a051}
 \begin{array}{c}
  \displaystyle{
 \Lambda^{a_2} Q^{a_1}=\exp\left(\frac{2\pi\imath}{{ M}}\,a_1
 a_2\right)Q^{a_1} \Lambda^{a_2}\,,\ a_{1,2}\in\mZ\,.
 }
 \end{array}
 \eq
 In particular, $T_0=T_{(0,0)}=1_M$ is the identity $M\times M$ matrix.
 For the product of basis matrices we have:
\beq\label{Tcond}\begin{array}{c} \displaystyle{
    T_\al T_\be = \ka_{\al, \be} T_{\al + \be}, \quad \ka_{\al, \be} = \exp \left( \frac{\pi i}{M}(\al_2 \be_1 - \al_1 \be_2) \right)\,,
}\end{array}\eq
where $\al+\be=(\al_1+\be_1,\al_2+\be_2)$. In particular, (\ref{Tcond}) means that $T_{-\al}=T_\al^{-1}$.


\paragraph{The Baxter-Belavin elliptic $R$-matrix \cite{Baxter,Belavin}.}
Using (\ref{a08}) and (\ref{a971}) introduce
\begin{equation}\label{BB}
\begin{array}{c}
    \displaystyle{
    R^{\hbar}_{12} (x) = \frac{1}{M}
    \sum_\al \varphi_\al (x, \frac{\hbar}{M} + \om_\al) T_\al \otimes T_{-\al}}\in\MatM^{\otimes 2}\,.
\end{array}
\end{equation}
This is the elliptic ($\mZ_M$ symmetric) Baxter-Belavin $R$-matrix. The $\mZ_M$ symmetry means
that
\begin{equation}\label{r081}
\begin{array}{c}
    \displaystyle{
    (Q\otimes Q) R^{\hbar}_{12} (x) = R^{\hbar}_{12} (x) (Q\otimes Q) \,,\qquad
    (\Lambda\otimes \Lambda) R^{\hbar}_{12} (x) = R^{\hbar}_{12} (x) (\Lambda\otimes \Lambda) \,.
    }
\end{array}
\end{equation}
This $R$-matrix is often written in different matrix basis and/or using some different
set of functions.  Alternative forms for (\ref{BB}) can be found in \cite{Belavin,Pol}.

By writing $R^{\hbar}_{ij} (x)\in\MatM^{\otimes N}$ (with $1\leq i,j\leq N$) it is assumed that the rest of components (except the $i$-th and the $j$-th) are filled by identity matrices:
\begin{equation}\label{BB2}
\begin{array}{c}
    \displaystyle{
    R^{\hbar}_{ij} (x) =
    }
    \\ \ \\
     \displaystyle{
    =\frac{1}{M}
    \sum_\al \varphi_\al (x, \frac{\hbar}{M} + \om_\al) 1_M\otimes \dots\otimes 1_M\otimes T_\al
     \otimes 1_M\otimes\dots \otimes 1_M\otimes T_{-\al}}\otimes 1_M\otimes \dots\otimes 1_M\,,
\end{array}
\end{equation}
where $T_\al$ is on the $i$-th place and $T_{-\al}$ is on the $j$-th place.
The $R$-matrix (\ref{BB}) is unitary as in (\ref{q03}) and satisfies both Yang-Baxter equations (\ref{QYB}) and (\ref{AYBE}).

Instead of the argument symmetry $\phi(x,z)=\phi(z,x)$, the $R$-matrix (\ref{BB}) satisfies the Fourier symmetry:
\begin{equation}\label{w33}\begin{array}{c}
    R^z_{12}(x) P_{12} = R^x_{12}(z)\,,
\end{array}\end{equation}
where $P_{12}$ is the permutation operator acting as $P_{12}(a\otimes b)=(b\otimes a)$ for any $a,b\in\mC^M$.
Also, for any $A,B\in\MatM$: $P_{12}(A\otimes B)=(B\otimes A)P_{12}$. Explicit form of $P_{12}$ is as follows:
\begin{equation}\label{P12}
\begin{array}{c}
    \displaystyle{
    P_{12}=\sum\limits_{k,l=1}^M e_{kl}\otimes e_{lk}=\frac{1}{M}\sum\limits_{\al\in\,\mZ_M\times\mZ_M}T_\al\otimes T_{-\al}\,,
    }
\end{array}
\end{equation}
where $\{e_{kl}\}$ is the standard matrix basis in $\MatM$. Obviously, $P_{12}^2=1_{M^2}$. The set of permutation operators $P_{ij}\in\MatM^{\otimes N}$ is defined similarly to the lift from (\ref{BB}) to (\ref{BB2}).
 The operators $P_{ij}$ satisfy the properties:
\begin{equation}\label{P12-1}
\begin{array}{c}
    \displaystyle{
    R^\hbar_{ij}(z)=P_{ij}R^\hbar_{ji}(z)P_{ij}\,, \qquad
    R^\hbar_{ij}(z)=P_{kj}R^\hbar_{ik}(z)P_{kj}
    }
\end{array}
\end{equation}
 for any distinct $i,j,k$.

Similarly to the scalar case $R$-matrix (\ref{BB}) is skew-symmetric
 \beq\label{r08}
 \begin{array}{c}
  \displaystyle{
 R^{\hbar}_{12}(z)=-R^{-\hbar}_{21}(-z)\,.
  }
 \end{array}
 \eq
 In both variables $\hbar$ or $z$ the $R$-matrix has simple poles and
 \beq\label{r05}
 \begin{array}{c}
  \displaystyle{
\res\limits_{z=0}R^{\hbar}_{12}(z)=
P_{12}\,,
  }
 \end{array}
 \eq
 \beq\label{r051}
 \begin{array}{c}
  \displaystyle{
\res\limits_{\hbar=0} R^{\hbar}_{12}(z)=1_M\otimes 1_M\,.
}
 \end{array}
 \eq
 The quasi-periodic behaviour on the lattice of periods is as follows:
 \beq\label{r721}
 \begin{array}{c}
  \displaystyle{
 R_{12}^\hbar(z+1)=(Q^{-1}\otimes 1_M)R_{12}^\hbar(z)(Q\otimes 1_M)\,,
  }
 \\ \ \\
  \displaystyle{
 R_{12}^\hbar(z+\tau)=\exp(-2\pi\imath\frac{\hbar}{M})\,(\Lambda^{-1}\otimes 1_M)R_{12}^\hbar(z)(\Lambda\otimes
 1_M)\,.
  }
 \end{array}
 \eq
 In the classical limit $\hbar\rightarrow 0$ we have the expansion:
 \beq\label{r052}
 \begin{array}{c}
  \displaystyle{
 R_{12}^\hbar(z)=\frac{1_M\otimes 1_M}{\hbar} +r_{12}(z)+O(\hbar)\,,
}
 \end{array}
 \eq
 where
 \beq\label{r053}
 \begin{array}{c}
    \displaystyle{
    r_{12} (z) = \frac{1}{M} E_1(z) 1_M \otimes 1_M + \frac{1}{M}
    \sum_{\al \neq 0} \varphi_\al (z, \om_\al) T_\al \otimes T_{-\al}
    }
\end{array}
\eq
 is the classical Belavin-Drinfeld elliptic $r$-matrix \cite{BD} satisfying the classical Yang-Baxter equation:
 \beq\label{r054}
 \begin{array}{c}
    \displaystyle{
    [r_{12},r_{23}]+[r_{12},r_{13}]+[r_{13},r_{23}]=0\,,
    \qquad r_{ij}=r_{ij}(z_i-z_j)\,.
    }
\end{array}
\eq
More properties of the elliptic $R$-matrix can be found in \cite{RT} and \cite{Pol,LOZ15,Z18}. In the Appendix of \cite{ZZ} different forms
of the elliptic $R$-matrix are reviewed.

\paragraph{$R$-matrix in $M=2$ case.} In this case $R^{\hbar}_{12} (z)$ is an element of ${\rm End}(\mC^2\otimes \mC^2)$, i.e. it is $4\times 4$ matrix known as the Baxter's $R$-matrix for eight-vertex model \cite{Baxter}. Namely,
in $M=2$ case the matrices $Q$, $\Lambda$ (\ref{a041}) take the form
 \beq\label{r722}
 \begin{array}{c}
  \displaystyle{
 Q=\mats{-1}{0}{0}{1}\,,\quad \Lambda=\mats{0}{1}{1}{0}
  }
 \end{array}
 \eq
and the basis matrices (\ref{a971}) are $T_{00}=1_2=\sigma_0$, $T_{10}=-\sigma_3$, $T_{10}=\sigma_1$ and
$T_{11}=\sigma_2$, where $\{\sigma_a\}$, $a=0,1,2,3,4$ are the Pauli matrices:
 \beq\label{r723}
 \begin{array}{c}
  \displaystyle{
 \sigma_0=\mats{1}{0}{0}{1}\,,\quad
 \sigma_1=\mats{0}{1}{1}{0}\,,\quad
 \sigma_2=\mats{0}{-\imath}{\imath}{0}\,,\quad
 \sigma_3=\mats{1}{0}{0}{-1}\,.

  }
 \end{array}
 \eq
The $R$-matrix (\ref{BB}) for $M=2$ has the form
 \beq\label{r724}
 \begin{array}{c}
  \displaystyle{
 R_{12}^\hbar(z)
 =\frac{1}{2}\Big(\vf_{00}\,\sigma_0\otimes\sigma_0
 +\vf_{01}\,\sigma_1\otimes\sigma_1
  +\vf_{11}\,\sigma_2\otimes\sigma_2
  +\vf_{10}\,\sigma_3\otimes\sigma_3\Big)\,,
  }
 \end{array}
\eq
 \beq\label{r826}
 \begin{array}{c}
  \displaystyle{
 \vf_{00}=\phi(z,\frac{\hbar}{2})\,,\quad
 \vf_{10}=\phi(z,\frac{1}{2}+\frac{\hbar}{2})\,,\quad
 \vf_{01}=e^{\pi\imath z}\phi(z,\frac{\tau}{2}+\frac{\hbar}{2})\,,\quad
 \vf_{11}=e^{\pi\imath z}\phi(z,\frac{1+\tau}{2}+\frac{\hbar}{2})\,.
 }
  \end{array}
 \eq
 In $4\times 4$ form it is as follows:
 \beq\label{r827}
 \begin{array}{c}
 R_{12}^\hbar(z)=\frac12\left(
 \begin{array}{cccc}
 \vf_{00}+\vf_{10} & 0 & 0 & \vf_{01}-\vf_{11}
 \\
 0 & \vf_{00}-\vf_{10} & \vf_{01}+\vf_{11} & 0
  \\
 0 & \vf_{01}+\vf_{11} & \vf_{00}-\vf_{10} & 0
 \\
 \vf_{01}-\vf_{11} & 0 & 0 & \vf_{00}+\vf_{10}
 \end{array}
 \right)
  \end{array}
 \eq
 This $R$-matrix satisfies (\ref{QYB}), (\ref{QYB2}) and the unitarity in the form (\ref{q03}).
 Divided by $\phi(z,\hbar)$ it turns into ${\bar R}_{12}^\hbar(z)$ (\ref{q04}), which satisfies (\ref{QYB})
 and the unitarity in the form (\ref{q05}), but does not satisfy (\ref{QYB2}).


\paragraph{Higher rank addition formula for $R$-matrices.}
Here we comment on the origin of (\ref{a4}).
The proof of these identities is based on our previous results.
The $N$-th order identity for the Kronecker function is known:
\beq\label{x1}
  \begin{array}{c}
  \displaystyle{
 \prod\limits_{i=1}^N \phi(x_i,y_i)=\sum\limits_{i=1}^N
 \phi \Bigl (x_i,\sum\limits_{l=1}^Ny_l \Bigr )\prod\limits_{j\neq
 i}^N\phi(x_j-x_i,y_j)\,.
 }
 \end{array}
 \eq
In \cite{Z18} the following $R$-matrix analogue of (\ref{x1}) was proved:
  \beq\label{e21}
  \begin{array}{l}
  \displaystyle{
R_{12}^{x_1}(y_1)R_{23}^{x_2}(y_2)\dots R_{N,N+1}^{x_N}(y_N)=
 }
 \\ \ \\
   \displaystyle{
 =R_{1,N+1}^{x_N}(Y_N)R_{12}^{x_1-x_N}(y_1)R_{23}^{x_2-x_N}(y_2)\dots
 R_{N-1,N}^{x_{N-1}-x_N}(y_{N-1})+
 }
  \\ \ \\
   \displaystyle{
 +R_{N,N+1}^{x_N-x_{N-1}}(y_N)R_{1,N+1}^{x_{N-1}}(Y_N)R_{12}^{x_1-x_{N-1}}(y_1)R_{23}^{x_2-x_{N-1}}(y_2)\dots
 R_{N-2,N-1}^{x_{N-2}-x_{N-1}}(y_{N-2})+
 }
 \end{array}
 \eq
$$
  \begin{array}{l}
   \displaystyle{
 +R_{N-1,N}^{x_{N-1}-x_{N-2}}(y_{N-1})R_{N,N+1}^{x_N-x_{N-2}}(y_N)R_{1,N+1}^{x_{N-2}}(Y_N)R_{12}^{x_1-x_{N-2}}(y_1)\dots
 R_{N-3,N-2}^{x_{N-3}-x_{N-2}}(y_{N-3})+
 }
    \\ \ \\
   \displaystyle{
 \ldots
 }
    \\ \ \\
   \displaystyle{
 +R_{23}^{x_2-x_1}(y_2)R_{34}^{x_3-x_1}(y_3)\dots R_{N,N+1}^{x_{N}-x_{1}}(y_{N})R_{1,N+1}^{x_{1}}(Y_N)
 \,,
 }
 \end{array}
 $$
where $Y_N=\sum\limits_{l=1}^N
y_l$. In compact form we have
 \beq\label{x2}
  \begin{array}{c}
  \displaystyle{
 \overrightarrow{\prod\limits_{i=1}^N}
 R_{i,i+1}^{x_i}(y_i)=
 \sum\limits_{i=1}^N
 \overrightarrow{\prod\limits_{j=N-i+2}^{N}}
 R_{j,j+1}^{x_j-x_{N-i+1}}(y_j)
 \cdot R_{1,N+1}^{x_{N-i+1}}\Bigl (\sum\limits_{l=1}^Ny_l \Bigr
 )\cdot
  \overrightarrow{\prod\limits_{j=1}^{N-i}}
 R_{j,j+1}^{x_j-x_{N-i+1}}(y_j)\,.

 }
 \end{array}
 \eq
Multiply both parts from the right by $P_{12}P_{13}...P_{1N}$ and using the Fourier symmetry (\ref{w33})
we come to (\ref{a4}).


\subsection{Appendix B: proof of Lemma {\ref{lemXYZ}}}\label{sec:B}
\def\theequation{B.\arabic{equation}}
\setcounter{equation}{0}
Here we give a sketch of a proof for (\ref{l21}), and (\ref{l22}) is proved similarly.

Notice that (\ref{l21}) is true if any of the subsets ($A$, $B$ or $C$) is empty. Suppose all of them are not empty. Consider the case when all the subsets consist of a single element (i.e. $|A|=|B|=|C|=1$): $A=\{a\}$, $B=\{b\}$ and $C=\{c\}$. If $b>a$ (or $c>b$) then $\mR_{B,A}={\rm Id}$ (or $\mR_{C,B}={\rm Id}$) and  (\ref{l21}) is again trivial. If $c<b<a$ then (\ref{l21}) is the quantum Yang-Baxter equation (\ref{QYB}).

In the general case the proof becomes a cumbersome due to a wide range of possible configurations (orderings) of indices on the interval $1,..,N$. For simplicity we consider the configurations satisfying
\beq\label{x10}
  \begin{array}{c}
  \displaystyle{
 1\leq c_1<...<c_{|C|}<b_1<...<b_{|B|}<a_1<...<a_{|A|}\leq N\,.
 }
 \end{array}
 \eq
The proof for a generic configuration of indices differs from the above just by multiple usage of (\ref{QYB3}).

Beginning with $|A|=|B|=|C|=1$ we are going to increase  the number of indices. Consider first an arbitrary $|C|$, while keeping   $|A|=|B|=1$. Then from (\ref{RIJ2}) we have
\beq\label{x11}
  \begin{array}{c}
  \displaystyle{
 \mR_{C,A\cup B}=R_{c_{|C|},b}R_{c_{|C|},a}\cdot R_{c_{|C|-1},b}R_{c_{|C|-1},a}\dots
 R_{c_{1},b}R_{c_{1},a}
 }
 \end{array}
 \eq
 and (since $\mR_{B,A}=R_{b,a}$)
\beq\label{x12}
  \begin{array}{c}
  \displaystyle{
 \mR_{C,A\cup B}\mR_{B,A}=R_{c_{|C|},b}R_{c_{|C|},a}\dots
 R_{c_{1},b}R_{c_{1},a}\cdot R_{b,a}=
 }
 \\ \ \\
   \displaystyle{
 =
 R_{c_{|C|},b}R_{c_{|C|},a}\dots
 R_{c_{2},b}R_{c_{2},a}\cdot R_{b,a}\cdot R_{c_{1},a}R_{c_{1},b}\,,
 }
 \end{array}
 \eq
 where we applied the Yang-Baxter equation (\ref{QYB}) as $R_{c_{1},b}R_{c_{1},a} R_{b,a}=R_{b,a} R_{c_{1},a}R_{c_{1},b}$.
 By repeating the same $|C|-1$ times more,  one easily obtains the r.h.s. of (\ref{l21}).

 Next, let us enlarge the number of indices in $B$ keeping $|A|=1$ (and $|C|$ is arbitrary). We prove
 (\ref{l21}) in this case by induction. Suppose it is true for some fixed $|B|$ with
\beq\label{x13}
  \begin{array}{c}
  \displaystyle{
 1\leq c_1<...<c_{|C|}<b_1<...<b_{|B|}<a\leq N\,.
 }
 \end{array}
 \eq
 Let us prove (\ref{l21}) for $B\rightarrow B\cup \{b'\}$ with $b_{|B|}<b'<a$. We begin with the r.h.s. of  (\ref{l21}).
 Since $\mR_{C\cup B\cup \{b'\},A}=\mR_{\{b'\},A}\mR_{C\cup B,A}$ and $\mR_{C,B\cup \{b'\}}=\mR_{C,B}\mR_{C,\{b'\}}$ we have
\beq\label{x14}
  \begin{array}{c}
  \displaystyle{
 \mR_{C\cup B\cup \{b'\},A}\mR_{C,B\cup \{b'\}}=\mR_{\{b'\},A}\mR_{C\cup B,A}\mR_{C,B}\mR_{C,\{b'\}}=
 \mR_{\{b'\},A} \mR_{C,A\cup B}\mR_{B,A} \mR_{C,\{b'\}}\,,
 }
 \end{array}
 \eq
 where we used the induction assumption. Since $A=\{a\}$ we also have $\mR_{C,A\cup B}=\mR_{C,B}\mR_{C,A}$. Then continue transforming (\ref{x14}) as follows:
\beq\label{x15}
  \begin{array}{c}
  \displaystyle{
 \mR_{\{b'\},A} \mR_{C,B}\mR_{C,A}\mR_{B,A} \mR_{C,\{b'\}}=
  \mR_{C,B} \mR_{C\cup\{b'\},A} \mR_{C,\{b'\}}\mR_{B,A}=
 }
 \\ \ \\
   \displaystyle{
  =\mR_{C,B} \mR_{C,A\cup\{b'\}} \mR_{\{b'\},A} \mR_{B,A}=
  \mR_{C,A\cup B\cup\{b'\}} \mR_{B\cup \{b'\},A}\,,
 }
 \end{array}
 \eq
 where in the beginning of the second line we again applied the induction assumption in the form $\mR_{C\cup\{b'\},A} \mR_{C,\{b'\}}=\mR_{C,A\cup\{b'\}} \mR_{\{b'\},A}$.

 Finally. let us enlarge the number of indices in $A$. Again, we do it by induction. Suppose
 (\ref{l21}) is true for some $A,B,C$ with indices as in (\ref{x10}). Let us prove it for $A\rightarrow A\cup\{a'\}$
 with $a_{|A|}<a'\leq N$. Here we begin with the l.h.s. of (\ref{l21}). Since $\mR_{C,A\cup\{a'\}\cup B}=\mR_{C,A\cup B}\mR_{C,\{a'\}}$ and $\mR_{B,A\cup\{a'\}}=\mR_{B,A}\mR_{B,\{a'\}}$ we have
\beq\label{x16}
  \begin{array}{c}
  \displaystyle{
 \mR_{C,A\cup\{a'\}\cup B}\mR_{B,A\cup\{a'\}}=\mR_{C,A\cup B}\mR_{C,\{a'\}}\mR_{B,A}\mR_{B,\{a'\}}=
 \mR_{C,A\cup B}\mR_{B,A}\mR_{C,\{a'\}}\mR_{B,\{a'\}}=
 }
 \\ \ \\
   \displaystyle{
   =\mR_{B\cup C,A}\mR_{C,B\cup\{a'\}} \mR_{B,\{a'\}}
   =\mR_{B\cup C,A}\mR_{B\cup C,\{a'\}}\mR_{C,B}=\mR_{B\cup C,A\cup\{a'\}}\mR_{C,B}\,.
 }
 \end{array}
 \eq
 This finishes the proof.


\subsection{Appendix C: proof of Lemma \ref{lprop1}}\label{sec:C}
\def\theequation{C.\arabic{equation}}
\setcounter{equation}{0}
Consider the expression (\ref{a21}) for ${\mathcal F}^+_{i_1,...,i_k}$. Our purpose is to calculate its
 residue at $z_a=z_b+\eta$ and transform it to the form (\ref{r752}). Let $I=\{i_1,...,i_k\}$ be the subset of $\{1,\dots,N\}$ and $I^c=\{1,\dots,N\}\setminus I$ is complement of a set $I$. Then ${\mathcal F}^+_{i_1,...,i_k}$ can be written in the notations (\ref{RIJsh1}) and (\ref{RIJsh2}) in the following way:
 \beq\label{q801}
{\mathcal F}^+_{i_1,...,i_k}=\mathcal{R}_{I,I^c}\cdot \mathcal{R}'_{I^c_-,I}\cdot \mathcal{R}_{I_-^c,I}\cdot \mathcal{R}'_{I,I^c}.
\eq
 If $a,b\in I$ or $a,b\in I^c$ the expression (\ref{q801}) does not contain a pole at $z_a=z_b+\eta$. We assume $b=i_d\in I$, $a\in I^c$ and $a<b$. The simple pole
 at $z_a=z_b+\eta$ comes from only $R_{ab}^-$ presenting in the $R_{I_-^c,I}$ (or equivalently in the middle line in the r.h.s. of (\ref{a21})).

 Denote by $J$ and $J^\bullet$ the subsets of $\{1,2,\dots,N\}$:
\beq\label{q803a}
J=I\setminus\{b\}\ ,\qquad \qquad J^\bullet=I^c\setminus\{a\}.
\eq
Let use rewrite $ {\mathcal F}_{i_1,...,i_{d-1},i_{d+1},...,l_k}^+(\{z_1,...,z_N\}\setminus\{z_a,z_b\}, k-1,N-2)$ in r.h.s of (\ref{r752}) in notations (\ref{RIJsh1}), (\ref{RIJsh2}):
\beq\label{q815}
{\mathcal F}_{i_1,...,i_{d-1},i_{d+1},...,i_k}^+(\{z_1,...,z_N\}\setminus\{z_a,z_b\}, k-1,N-2)=\mathcal{R}_{J,J^\bullet}\cdot\mathcal{R}'_{J^\bullet_-,J}\cdot \mathcal{R}_{J^\bullet_-,J}\cdot\mathcal{R}'_{J,J^c}.
\eq
Indeed, due to (\ref{q803a}) we have $J=\{i_1,...,i_{d-1},i_{d+1},...,i_k\}$ and $J^\bullet=\{1,\dots,N\}\setminus J\setminus \{a,b\}$. So $J^\bullet$ is a complement set of $J$ to set with $(N-2)$ elements.
Our aim is to prove the formula (\ref{r752}) given in these notations as
\beq\label{q816}
 \begin{array}{c}
  \displaystyle{
  \res\limits_{z_a=z_b+\eta}{\mathcal F}_{i_1,...,i_k}
  ={\mathcal A}(a,b)\cdot\mathcal{R}_{J,J^\bullet}\cdot\mathcal{R}'_{J^\bullet_-,J}\cdot \mathcal{R}_{J^\bullet_-,J}\cdot\mathcal{R}'_{J,J^c}
  \cdot P_{ab} \cdot{\mathcal B}(a,b)\,.
  }
 \end{array}
 \eq

 The calculation of the residue $\res\limits_{z_a=z_b+\eta}{\mathcal F}^+_{i_1,...,i_k}$ is simple. Due to
 (\ref{r05}) one should just replace $R_{ab}^-$ with $P_{ab}$ and restrict the obtained expression to $z_a=z_b+\eta$.

 The main part of the proof is to transform the residue to the form (\ref{q816}). The idea is as follows. First, using lemma (\ref{lemXYZ}) we separate the $R$-matrices in (\ref{q801})  containing tensor index $a$ or $b$ and the others.
 Then we move the permutation operator $P_{ab}$ to the right using the properties (\ref{P12-1}).

Firstly, we rewrite (\ref{q801}) using the notation (\ref{pI}):
 \beq\label{q803}
{\mathcal F}^+_{i_1,...,i_k}=\mathcal{R}_{I,I^c}\cdot\mathbf{p}_{I^c}\cdot \mathcal{R}'_{I^c,I}\cdot \mathcal{R}_{I^c,I}\cdot\mathbf{p}_{I^c}^{-1}\cdot\mathcal{R}'_{I,I^c}.
\eq
In order to separate $\mathcal{R}_{J,J^\bullet}$ and the $R$-matrices with tensor index $a$ or $b$ we use Lemma \ref{lemXYZ}:
$$
\begin{array}{c}
\mathcal{R}_{I,I^c}=\mathcal{R}_{J\cup\{b\},J^\bullet\cup\{a\}}=\mathcal{R}_{J\cup\{b\},J^\bullet\cup\{a\}}\cdot\mathcal{R}_{\{b\},J}\cdot\left(\mathcal{R}_{\{b\},J}\right)^{-1}
\\ \ \\
=\mathcal{R}_{\{b\},J\cup J^\bullet\cup\{a\}}\cdot\mathcal{R}_{J,J^\bullet\cup\{a\}}\cdot\mathcal{R}_{J^\bullet,\{a\}}\cdot\left(\mathcal{R}_{J^\bullet,\{a\}}\right)^{-1}\left(\mathcal{R}_{\{b\},J}\right)^{-1}=
\end{array}
$$
\beq\label{q804}
=\mathcal{R}_{\{b\},J\cup J^\bullet\cup\{a\}}\cdot \mathcal{R}_{J\cup J^\bullet,\{a\}}\cdot \mathcal{R}_{J,J^\bullet}\cdot \left(\mathcal{R}_{J^\bullet,\{a\}}\right)^{-1}\cdot \left(\mathcal{R}_{\{b\},J}\right)^{-1}
\eq
In the second equality we insert identity operator ${\rm Id}=\mathcal{R}_{\{b\},J}\cdot\left(\mathcal{R}_{\{b\},J}\right)^{-1}$ to the right and in the next equality use (\ref{l21}) for the first two items, then we insert ${\rm Id}=\mathcal{R}_{J^\bullet,\{a\}}\cdot\left(\mathcal{R}_{J^\bullet,\{a\}}\right)^{-1}$ before $\left(\mathcal{R}_{\{b\},J}\right)^{-1}$, in the last equality we use (\ref{l21}) for the second and third items. Similar transformations can be done with other items of (\ref{q803}):

\beq\label{q805}
\mathcal{R}_{I^c,I}=\mathcal{R}_{\{a\},J\cup J^\bullet\cup\{b\}}\cdot \mathcal{R}_{J\cup J^\bullet,\{b\}}\cdot \mathcal{R}_{J^\bullet,J}\cdot \left(\mathcal{R}_{J,\{b\}}\right)^{-1}\cdot \left(\mathcal{R}_{\{a\},J^\bullet}\right)^{-1}\,,
\eq
\beq\label{q806}
\mathcal{R}'_{I,I^c}=\left(\mathcal{R}'_{J^\bullet,\{a\}}\right)^{-1}\cdot \left(\mathcal{R}'_{\{b\},J}\right)^{-1}\cdot \mathcal{R}'_{J,J^\bullet}\cdot \mathcal{R}'_{\{b\},J\cup J^\bullet}\cdot \mathcal{R}'_{J\cup J^\bullet\cup\{b\},\{a\}}\,,
\eq
\beq\label{q807}
\mathcal{R}'_{I^c,I}= \left(\mathcal{R}'_{J,\{b\}}\right)^{-1}\cdot \left(\mathcal{R}'_{\{a\},J^\bullet}\right)^{-1}\cdot \mathcal{R}'_{J^\bullet,J}\cdot \mathcal{R}'_{\{a\},J\cup J^\bullet}\cdot \mathcal{R}'_{J\cup J^\bullet\cup\{a\},\{b\}}\,.
\eq

Let us substitute expressions (\ref{q804}),  (\ref{q805}),  (\ref{q806}),  (\ref{q807}) in (\ref{q803}):
\beq\label{q808}
\begin{array}{c}
 {\mathcal F}^+_{i_1,...,i_k}=
 \\ \ \\
 =\mathcal{R}_{\{b\},J\cup J^\bullet\cup\{a\}}\cdot \mathcal{R}_{J\cup J^\bullet,\{a\}}\cdot \mathcal{R}_{J,J^\bullet}\cdot \left(\mathcal{R}_{J^\bullet,\{a\}}\right)^{-1}\cdot \left(\mathcal{R}_{\{b\},J}\right)^{-1} \mathbf{p}_{I^c}\left(\mathcal{R}'_{J,\{b\}}\right)^{-1}\cdot \left(\mathcal{R}'_{\{a\},J^\bullet}\right)^{-1}
 \\ \ \\
 \mathcal{R}'_{J^\bullet,J}\cdot \mathcal{R}'_{\{a\},J\cup J^\bullet}\cdot \mathcal{R}'_{J\cup J^\bullet\cup\{a\},\{b\}}\cdot\mathcal{R}_{\{a\},J\cup J^\bullet\cup\{b\}}\cdot \mathcal{R}_{J\cup J^\bullet,\{b\}}\cdot \mathcal{R}_{J^\bullet,J}\cdot
 \\ \ \\
 \left(\mathcal{R}_{J,\{b\}}\right)^{-1}\cdot \left(\mathcal{R}_{\{a\},J^\bullet}\right)^{-1}\mathbf{p}_{I^c}^{-1}\left(\mathcal{R}'_{J^\bullet,\{a\}}\right)^{-1}\cdot \left(\mathcal{R}'_{\{b\},J}\right)^{-1}\cdot \mathcal{R}'_{J,J^\bullet}\cdot \mathcal{R}'_{\{b\},J\cup J^\bullet}\cdot \mathcal{R}'_{J\cup J^\bullet\cup\{b\},\{a\}}\,.
\end{array}
\eq
The first two items and the last two items of (\ref{q808}) are $\mathcal{A}(a,b)$ and $\mathcal{B}(a,b)$ from (\ref{r753}), (\ref{r754}), respectively. Indeed, due to definitions (\ref{RIJsh1}) and (\ref{RIJsh2})::

\beq\label{q809}
\mathcal{A}(a,b)=\mathcal{R}_{\{b\},J\cup J^\bullet\cup\{a\}}\cdot \mathcal{R}_{J\cup J^\bullet,\{a\}}= R_{b,b+1}...R_{b,N}\cdot R_{a-1,a}...R_{1,a}\, ,
\eq
\beq\label{q810}
\mathcal{B}(a,b)=\mathcal{R}'_{\{b\},J\cup J^\bullet}\cdot \mathcal{R}'_{J\cup J^\bullet\cup\{b\},\{a\}}=R_{b,1}\dots R_{b,a-1}\cdot R_{b,a+1}\dots R_{b-1,b}\cdot R_{N,a}\dots R_{a+1,a}\,.
\eq
Note that $\mathbf{p}_{I^c}$ commute with $\left(\mathcal{R}'_{J,\{b\}}\right)^{-1}$ and $\left(\mathcal{R}'_{\{a\},J^\bullet}\right)^{-1}$. Indeed, $\mathbf{p}_{I^c}$ shifts only $z_a$ and $z_{j^\bullet}$,  $j^\bullet\in J^\bullet$, on which  the first item does not depend and the second depends on difference $(z_a-z_{j^\bullet})$. We move  $\left(\mathcal{R}'_{J,\{b\}}\right)^{-1}\cdot\left(\mathcal{R}'_{\{a\},J^\bullet}\right)^{-1}$ to the left and put it before $\mathbf{p}_{I^c}$. Then we use (\ref{un01}) to reduce the $R$-matrices:
\beq
\left(\mathcal{R}_{\{b\},J}\right)^{-1}\left(\mathcal{R}'_{J,\{b\}}\right)^{-1}=\left(\prod\limits_{\substack{j\in J\\j>b}} \phi(z_j-z_b)\phi(z_b-z_j)\right)^{-1}
\eq
The same should be done for $\left(\mathcal{R}_{J,\{b\}}\right)^{-1}\cdot \left(\mathcal{R}_{\{a\},J^\bullet}\right)^{-1}\mathbf{p}_{I^c}^{-1}\left(\mathcal{R}'_{J^\bullet,\{a\}}\right)^{-1}\cdot \left(\mathcal{R}'_{\{b\},J}\right)^{-1}$ in the last line of (\ref{q808}).
The expression (\ref{q808}) becomes
\beq\label{q811}
\begin{array}{c}
 {\mathcal F}^+_{i_1,...,i_k}=\mathcal{A}(a,b) \mathcal{R}_{J,J^\bullet}\left( \prod\limits_{\substack{j^\bullet \in J^\bullet\\j^\bullet <a}} \phi(z_{j^\bullet}-z_a)\phi(z_a-z_{j^\bullet})\prod\limits_{\substack{j\in J\\j>b}} \phi(z_j-z_b)\phi(z_b-z_j)\right)^{-1}
 \\ \ \\
  \mathbf{p}_{I^c}\left(\mathcal{R}'_{J^\bullet,J}\cdot \mathcal{R}'_{\{a\},J\cup J^\bullet}\cdot \mathcal{R}'_{J\cup J^\bullet\cup\{a\},\{b\}}\cdot\mathcal{R}_{\{a\},J\cup J^\bullet\cup\{b\}}\cdot \mathcal{R}_{J\cup J^\bullet,\{b\}}\cdot \mathcal{R}_{J^\bullet,J}\right)\mathbf{p}_{I^c}^{-1}\cdot
 \\ \ \\
 \left( \prod\limits_{\substack{j^\bullet \in J^\bullet\\j^\bullet >a}} \phi(z_{j^\bullet}-z_a)\phi(z_a-z_{j^\bullet})\prod\limits_{\substack{j\in J\\j<b}} \phi(z_j-z_b)\phi(z_b-z_j)\right)^{-1} \mathcal{R}'_{J,J^\bullet}\mathcal{B}(a,b)\,.
\end{array}
\eq

Let us write down two middle items in the middle line of (\ref{q811}) and transform it using Lemma \ref{lemXYZ}:
\beq\label{q812}
\begin{array}{c}
\mathcal{R}'_{J\cup J^\bullet\cup\{a\},\{b\}}\cdot\mathcal{R}_{\{a\},J\cup J^\bullet\cup\{b\}}=
 \\ \ \\
=\left(\mathcal{R}'_{J\cup J^\bullet,\{a\}}\right)^{-1}\cdot \mathcal{R}'_{J\cup J^\bullet,\{a\}} \cdot\mathcal{R}'_{J\cup J^\bullet\cup\{a\},\{b\}}\cdot\mathcal{R}_{\{a\},J\cup J^\bullet\cup\{b\}}\cdot\mathcal{R}_{\{b\},J\cup J^\bullet}\cdot\left( \mathcal{R}_{\{b\},J\cup J^\bullet}\right)^{-1}=
 \\ \ \\
 =\left(\mathcal{R}'_{J\cup J^\bullet,\{a\}}\right)^{-1}\cdot \mathcal{R}'_{\{a\},\{b\}} \cdot\mathcal{R}'_{J\cup J^\bullet,\{a,b\}}\cdot\mathcal{R}_{\{a,b\},J\cup J^\bullet}\cdot\mathcal{R}_{\{a\},\{b\}}\cdot\left( \mathcal{R}_{\{b\},J\cup J^\bullet}\right)^{-1}=
 \\ \ \\
 \displaystyle{
=\prod\limits_{m=b+1}^N \phi(z_m-z_b)\phi(z_b-z_m)\prod\limits_{\substack{l=a+1\\l\neq b}}^N \phi(z_l-z_a)\phi(z_a-z_l)\times}
\\
\times\left(\mathcal{R}'_{J\cup J^\bullet,\{a\}}\right)^{-1}\cdot R^\hbar_{ab}(z_a-z_b)\cdot\left( \mathcal{R}_{\{b\},J\cup J^\bullet}\right)^{-1}.
\end{array}
\eq
In first equality we insert identical operators  ${\rm Id}=\left(\mathcal{R}'_{J\cup J^\bullet,\{a\}}\right)^{-1}\cdot \mathcal{R}'_{J\cup J^\bullet,\{a\}}$ to the left and ${\rm Id}=\mathcal{R}_{\{b\},J\cup J^\bullet}\cdot\left( \mathcal{R}_{\{b\},J\cup J^\bullet}\right)^{-1}$ to the right. Then using  (\ref{l21}) and (\ref{l22}) we separate $\mathcal{R}'_{\{a\},\{b\}}=1$ and $\mathcal{R}_{\{a\},\{b\}}=R^\hbar_{ab}(z_a-z_b)$. In the last equality due to (\ref{un02}) the item  $\mathcal{R}'_{J\cup J^\bullet,\{a,b\}}\cdot\mathcal{R}_{\{a,b\},J\cup J^\bullet}$ is reduced.

Substitute the result of (\ref{q812}) into (\ref{q811}). Then, by unifying the products with  Kronecker functions in  (\ref{q811}) and moving them at the beginning one obtains:
\beq\label{q813}
\begin{array}{c}
 {\mathcal F}^+_{i_1,...,i_k}=\left( \prod\limits_{\substack{j^\bullet \in J^\bullet}} \phi(z_{j^\bullet}-z_a)\phi(z_a-z_{j^\bullet})\prod\limits_{\substack{j\in J}} \phi(z_j-z_b)\phi(z_b-z_j)\right)^{-1} \mathcal{A}(a,b) \mathcal{R}_{J,J^\bullet}\times
 \\ \ \\
 \times \mathbf{p}_{I^c}\left(\prod\limits_{m=b+1}^N \phi(z_m-z_b)\phi(z_b-z_m)\prod\limits_{\substack{l=a+1\\l\neq b}}^N \phi(z_l-z_a)\phi(z_a-z_l)\,\mathcal{R}'_{J^\bullet,J}\cdot \mathcal{R}'_{\{a\},J\cup J^\bullet}\times\right.
  \\ \ \\
  \left.\times\left(\mathcal{R}'_{J\cup J^\bullet,\{a\}}\right)^{-1}\cdot R^\hbar_{ab}(z_a-z_b)\cdot\left( \mathcal{R}_{\{b\},J\cup J^\bullet}\right)^{-1}\cdot \mathcal{R}_{J\cup J^\bullet,\{b\}}\cdot\mathcal{R}_{J^\bullet,J}\right)\mathbf{p}_{I^c}^{-1}\cdot \mathcal{R}'_{J,J^\bullet}\mathcal{B}(a,b)\,.
\end{array}
\eq

Next step is to find $\res\limits_{z_a=z_b+\eta}{\mathcal F}^+_{i_1,...,i_k}$ and transform it to (\ref{q816}). Taking the residue of (\ref{q813}):
\beq\label{q820}
\begin{array}{c}
 \res\limits_{z_a=z_b+\eta}{\mathcal F}^+_{i_1,...,i_k}=\left( \prod\limits_{\substack{j^\bullet \in J^\bullet}} \phi(z_{j^\bullet}-z_a)\phi(z_a-z_{j^\bullet})\prod\limits_{\substack{j\in J}} \phi(z_j-z_b)\phi(z_b-z_j)\right)^{-1} \mathcal{A}(a,b) \mathcal{R}_{J,J^\bullet}\times
 \\ \ \\
\times \mathbf{p}_{I^c}\left(\prod\limits_{m=b+1}^N \phi(z_m-z_b)\phi(z_b-z_m)\prod\limits_{\substack{l=a+1\\l\neq b}}^N \phi(z_l-z_a)\phi(z_a-z_l)\,\mathcal{R}'_{J^\bullet,J}\times\right.
  \\ \ \\
  \left.\times \mathcal{R}'_{\{a\},J\cup J^\bullet}\cdot\left(\mathcal{R}'_{J\cup J^\bullet,\{a\}}\right)^{-1}\cdot P_{ab}\cdot\left( \mathcal{R}_{\{b\},J\cup J^\bullet}\right)^{-1}\cdot\mathcal{R}_{J\cup J^\bullet,\{b\}}\cdot \mathcal{R}_{J^\bullet,J}\right)\mathbf{p}_{I^c}^{-1}\cdot \mathcal{R}'_{J,J^\bullet}\mathcal{B}(a,b)
\end{array}
\eq
we need to move $P_{ab}$ to the right. The permutation $P_{ab}$ commutes with  $\mathcal{R}_{J^\bullet,J}$ since $J$ and $J^\bullet$ do not contain indices $a,b$. Moving  $P_{ab}$ to the right we change each index $b$ to $a$, in details:
\beq\label{q821}
P_{ab}\left(\mathcal{R}_{\{b\},J\cup J^\bullet}\right)^{-1}=P_{ab} R_{b,N}^{-1}\cdot R_{b,N-1}^{-1}\dots R_{b,b+1}^{-1}= R_{a,N}^{-1}\cdot R_{a,N-1}^{-1}\dots R_{a,b+1}^{-1}P_{ab}\,,
\eq
\beq\label{q822}
P_{ab}\mathcal{R}_{J\cup J^\bullet,\{b\}}=P_{ab} R_{b-1,b}\dots R_{a+1,b}\cdot R_{a-1,b}\dots R_{1,b}= R_{b-1,a}\dots R_{a+1,a}\cdot R_{a-1,a}\dots R_{1,a}P_{ab}\,.
\eq
Notice that (\ref{q821})  is written in terms of short notations $R_{ij}=R_{ij}^\hbar(z_i-z_j)$ , so that any $R$-matrix with ''tensor indices'' $ij$ depends on $z_i-z_j$. When moving $P_{ab}$ to the right $R$-matrix indices change and the latter rule fails. But it can be restored by substituting either $z_b\rightarrow z_a$. Indeed, we take the residue at $z_a=z_b+\eta$ and the item $P_{ab}\left(\mathcal{R}_{\{b\},J\cup J^\bullet}\right)^{-1}$ is placed between $\mathbf{p}_{I^c}$ and $\mathbf{p}_{I^c}^{-1}$ in (\ref{q820}), so the substitution $z_b\rightarrow z_a$ is equivalent to taking the residue.

Let us write in details the two items before $P_{ab}$ in (\ref{q820}):
\beq\label{q823}
\mathcal{R}'_{\{a\},J\cup J^\bullet}=R_{a,1}\dots R_{a,a-1}\,,
\eq
\beq\label{q824}
\left(\mathcal{R}'_{J\cup J^\bullet,\{a\}}\right)^{-1}=R_{a+1,a}^{-1}\dots R_{b-1,a}^{-1}\cdot R_{b+1,a}^{-1}\dots R_{N,a}^{-1}\,.
\eq
Multiplying expressions (\ref{q821}), (\ref{q822}), (\ref{q823}), (\ref{q824}) in order given in (\ref{q820}) one should reduce the $R$-matrices:
\beq\label{q8241}
\begin{array}{c}
  \mathcal{R}'_{\{a\},J\cup J^\bullet}\cdot\left(\mathcal{R}'_{J\cup J^\bullet,\{a\}}\right)^{-1}\cdot P_{ab}\cdot\left( \mathcal{R}_{\{b\},J\cup J^\bullet}\right)^{-1}\cdot\mathcal{R}_{J\cup J^\bullet,\{b\}}=
  \\ \ \\
  =R_{a,1}\dots R_{a,a-1}\cdot R_{a+1,a}^{-1}\dots R_{b-1,a}^{-1}\cdot
  \hspace{9cm}\
  \\ \ \\
  \quad\ \quad\  \cdot\underbrace{R_{b+1,a}^{-1}\dots R_{N,a}^{-1}\cdot R_{a,N}^{-1}\dots R_{a,b+1}^{-1}}\cdot R_{b-1,a}\dots R_{a+1,a}\cdot R_{a-1,a}\dots R_{1,a}=
  \\ \ \\
  =\left(\prod_{m=b+1}^N \phi(z_m-z_a)\phi(z_a-z_m)\right)^{-1} R_{a,1}\dots R_{a,a-1}\cdot
    \hspace{6cm}\
    \\ \ \\
  \qquad\ \qquad\  \qquad\ \qquad\ \cdot  \underbrace{ R_{a+1,a}^{-1}\dots R_{b-1,a}^{-1}\cdot R_{b-1,a}\dots R_{a+1,a}}\cdot R_{a-1,a}\dots R_{1,a}=
   \end{array}
 \eq
\beq\label{q825}
\begin{array}{c}
  {\displaystyle
  =\left(\prod_{m=b+1}^N \phi(z_m-z_a)\phi(z_a-z_m)\right)^{-1}\, R_{a,1}\dots R_{a,a-1}\cdot R_{a-1,a}\dots R_{1,a}
  }=
   \\ \ \\
  {\displaystyle
  =\left(\prod_{m=b+1}^N \phi(z_m-z_a)\phi(z_a-z_m)\right)^{-1}\,\prod_{l=1}^{a-1}\phi(z_l-z_a)\phi(z_a-z_l)\,.
  }
  \end{array}
  \eq
  Here we use the unitarity property (\ref{q03}) for highlighted $R$-matrices in the second and the last equality. Taking into account (\ref{q825}) we collect all items with Kronecker functions in (\ref{q820}) and show that all these items vanish:
  \beq
  \begin{array}{c}
\left( \prod\limits_{\substack{j^\bullet \in J^\bullet}} \phi(z_{j^\bullet}-z_a)\phi(z_a-z_{j^\bullet})\prod\limits_{\substack{j\in J}} \phi(z_j-z_b)\phi(z_b-z_j)\right)^{-1} \bigg|_{z_b=z_a-\eta} \times
\\ \ \\
{\displaystyle
 \times \mathbf{p}_{I^c}\ \
 \frac{\prod\limits_{m=b+1}^N \phi(z_m-z_b)\phi(z_b-z_m)\prod\limits_{\substack{l=1\\l\neq b}}^N \phi(z_l-z_a)\phi(z_a-z_l)}{\displaystyle \prod_{m=b+1}^N \phi(z_m-z_a)\phi(z_a-z_m)}\bigg|_{z_b=z_a}\ \
 \mathbf{p}_{I^c}^{-1} =
 }
 \end{array}
 \eq
  \beq
  \begin{array}{c}
 {\displaystyle
 =\left( \prod\limits_{\substack{j^\bullet \in J^\bullet}} \phi(z_{j^\bullet}-z_a)\phi(z_a-z_{j^\bullet})\prod\limits_{\substack{j\in J}} \phi(z_j-z_a+\eta)\phi(z_a-\eta-z_j)\right)^{-1} \times}
 \\ \ \\
 {\displaystyle
 \times \mathbf{p}_{I^c}\ \ \prod\limits_{\substack{l=1\\l\neq b}}^N \phi(z_l-z_a)\phi(z_a-z_l)\ \  \mathbf{p}_{I^c}^{-1}=1\,.
 }
\end{array}
  \eq
  Indeed,  if $l\in J^{\bullet}$, $\mathbf{p}_{I^c}$ shifts both variables $z_a$ and $z_l$, thus the difference $(z_a-z_l)$ does not change, if $l\in J$, then  $(z_a-z_l)\to (z_a-z_l-\eta)$.

  We have shown that ({\ref{q820}}) equals (\ref{q816}), this finishes the proof of (\ref{r752}).  The proof of (\ref{r755}) is performed in a similar way. $\blacksquare$


\subsection{Appendix D: proof of Proposition {\ref{prop2}}}\label{sec:D}
\def\theequation{D.\arabic{equation}}
\setcounter{equation}{0}

Let us first compute the residue $\res\limits_{z_a=z_b+\eta-\Omega_{m_1,m_2}}R_{ab}^-$. For this purpose we need the transformation properties (\ref{r721}). It is easy to show that
 \beq\label{a901}
  \begin{array}{c}
     \displaystyle{
   R_{12}^\hbar(z-\Omega_{m_1,m_2})=\exp(2\pi\imath\frac{m_2\hbar}{M})\,(T_{m_1,m_2}\otimes 1_M)R_{12}^\hbar(z)(T_{m_1,m_2}^{-1}\otimes
 1_M)\,,
         }
      \end{array}
 \eq
    where $T_{m_1,m_2}$ is the basis matrix (\ref{a971}). Then
 \beq\label{a902}
  \begin{array}{c}
     \displaystyle{
   R_{ij}^\hbar(z-\Omega_{m_1,m_2})=\exp(2\pi\imath\frac{m_2\hbar}{M})\,T^{(i)}_{m_1,m_2}\, R_{ij}^\hbar(z)\,\Big(T^{(i)}_{m_1,m_2}\Big)^{-1}\,,
         }
      \end{array}
 \eq
    where
 \beq\label{a903}
  \begin{array}{c}
     \displaystyle{
   T^{(i)}_{m_1,m_2}=1_M\otimes\ldots\otimes 1_M\otimes T_{m_1,m_2}\otimes 1_M\otimes\ldots \otimes 1_M\in\MatM^{\otimes N}
         }
      \end{array}
 \eq
    with $T_{m_1,m_2}$ standing in the $i$-th tensor component. Thus, from (\ref{a902}) and (\ref{r05}) we conclude
 \beq\label{a9030}
  \begin{array}{c}
     \displaystyle{
    \res\limits_{z_a=z_b+\eta-\Omega_{m_1,m_2}}R_{ab}^-=\exp(2\pi\imath\frac{m_2\hbar}{M})\,T^{(a)}_{m_1,m_2}\, P_{ab}\,\Big(T^{(a)}_{m_1,m_2}\Big)^{-1}\,.
           }
      \end{array}
 \eq

    Next, let us represent the transformation properties with respect to a variable $z_a$ in a more universal form. It follows from the $\mZ_M$ symmetry (\ref{r081}) that
 \beq\label{a9031}
  \begin{array}{c}
     \displaystyle{
   R_{ij}^\hbar(z)T^{(i)}_{m_1,m_2}T^{(j)}_{m_1,m_2}=
   T^{(i)}_{m_1,m_2}T^{(j)}_{m_1,m_2}R_{ij}^\hbar(z)\,.
   }
      \end{array}
 \eq
    Therefore, we have
 \beq\label{a904}
  \begin{array}{c}
     \displaystyle{
   R_{ij}^\hbar(z+\Omega_{m_1,m_2})=\exp(-2\pi\imath\frac{m_2\hbar}{M})\,\Big(T^{(i)}_{m_1,m_2}\Big)^{-1}\, R_{ij}^\hbar(z)\,T^{(i)}_{m_1,m_2}=
   }
   \\ \ \\
     \displaystyle{
   =\exp(-2\pi\imath\frac{m_2\hbar}{M})\,T^{(j)}_{m_1,m_2}\, R_{ij}^\hbar(z)\,\Big(T^{(j)}_{m_1,m_2}\Big)^{-1}\,.
         }
      \end{array}
 \eq
By combining (\ref{a902}) and (\ref{a904}) we come to the following property. For any $R_{ij}=R^\hbar_{ij}(z_i-z_j)$ (or $R_{ij}^-=R^\hbar_{ij}(z_i-z_j-\eta)$)
 \beq\label{a905}
  \begin{array}{c}
     \displaystyle{
   R_{ij}\Big|_{z_a\rightarrow z_a-\Omega_{m_1,m_2}}=\exp\Big(2\pi\imath(\delta_{a,i}-\delta_{a,j})\frac{m_2\hbar}{M}\Big)\,
   T^{(a)}_{m_1,m_2}\, R_{ij}\,\Big(T^{(a)}_{m_1,m_2}\Big)^{-1}\,.
   }
      \end{array}
 \eq
Indeed, if both indices $i,j$ do not equal to $a$ then the r.h.s. of (\ref{a905}) is equal to $R_{ij}$.
If one of the indices (either $i$ or $j$) is equal to $a$ then in the r.h.s. of (\ref{a905}) we obtain either
(\ref{a902}) or (\ref{a904}).

Returning back to the proof of the Proposition \ref{prop2} we mention that
 \beq\label{a906}
  \begin{array}{c}
     \displaystyle{
   \res\limits_{z_a=z_b+\eta-\Omega_{m_1,m_2}} \mF=
   \res\limits_{z_a=z_b+\eta} \Big( \mF|_{z_a\rightarrow z_a-\Omega_{m_1,m_2}} \Big)
   }
      \end{array}
 \eq
From (\ref{a905}) we have
 \beq\label{a907}
  \begin{array}{c}
     \displaystyle{
    \mF|_{z_a\rightarrow z_a-\Omega_{m_1,m_2}}=T^{(a)}_{m_1,m_2}\,  \mF\,\Big(T^{(a)}_{m_1,m_2}\Big)^{-1}\,,
   }
      \end{array}
 \eq
    where the exponential factor is absent since in each term of the expression $\mF$ the number of $R$-matrices $R_{ij}$ (or $R^-_{ij}$) with $i=a$
is equal to the number of $R$-matrices $R_{ij}$ (or $R^-_{ij}$) with $j=a$. Plugging (\ref{a907}) into (\ref{a906}) yields
 \beq\label{a908}
  \begin{array}{c}
     \displaystyle{
   \res\limits_{z_a=z_b+\eta-\Omega_{m_1,m_2}} \mF=
   T^{(a)}_{m_1,m_2}\,\Big(\res\limits_{z_a=z_b+\eta} \mF \Big)\,\Big(T^{(a)}_{m_1,m_2}\Big)^{-1}=0\,.
   }
      \end{array}
 \eq
    The latter follows from the Proposition \ref{prop1}. $\blacksquare$



\subsection*{Acknowledgments}


We are grateful to J. Lamers, A. Liashyk, I. Sechin, V. Sokolov and A. Zabrodin for useful discussions.


This work was performed at the Steklov International Mathematical Center and supported by the Ministry of Science and Higher Education of the Russian Federation (agreement no. 075-15-2022-265).


\begin{small}

\end{small}

\end{document}